\documentclass[11pt,tbtags,draft]{amsart}
\usepackage{amssymb}
\usepackage{url}
\usepackage[square,numbers]{natbib}
\bibpunct[, ]{[}{]}{;}{n}{,}{,}

\numberwithin{equation}{section}

\renewcommand\le{\leqslant}
\renewcommand\ge{\geqslant}

\allowdisplaybreaks

\newtheorem{theorem}{Theorem}[section]
\newtheorem{lemma}[theorem]{Lemma}

\newtheorem{corollary}[theorem]{Corollary}

\theoremstyle{definition}
\newtheorem{example}[theorem]{Example}
\newtheorem{definition}[theorem]{Definition}

\newtheorem{remark}[theorem]{Remark}

\theoremstyle{remark}

\newenvironment{romenumerate}[1][0pt]{
\addtolength{\leftmargini}{#1}\begin{enumerate}
 }{\end{enumerate}}

\newenvironment{stairenumerate}[1][0pt]{
\addtolength{\leftmargini}{#1}\begin{enumerate}
 }{\end{enumerate}}

\newcounter{oldenumi}
{\setcounter{oldenumi}{\value{enumi}}
\begin{romenumerate} \setcounter{enumi}{\value{oldenumi}}}
{\end{romenumerate}}

\newcounter{thmenumerate}

\newcounter{romxenumerate}   

\newcounter{xenumerate}   

\newcommand\pfitemx[1]{\par#1:}
\newcommand\pfitemref[1]{\pfitemx{\ref{#1}}}

\newcommand{\refT}[1]{Theorem~\ref{#1}}

\newcommand{\refL}[1]{Lemma~\ref{#1}}
\newcommand{\refR}[1]{Remark~\ref{#1}}
\newcommand{\refS}[1]{Section~\ref{#1}}

\newcommand{\refD}[1]{Definition~\ref{#1}}
\newcommand{\refE}[1]{Example~\ref{#1}}
\newcommand{\refF}[1]{Figure~\ref{#1}}

\newcommand{\refTab}[1]{Table~\ref{#1}}
\newcommand{\refand}[2]{\ref{#1} and~\ref{#2}}

\newenvironment{comment}{\setbox0=\vbox\bgroup}{\egroup} 

\begingroup
  \divide\count255 by 60
  \multiply\count255 by -60
  \advance\count255 by \time
  \ifnum \count255 < 10 \xdef\klockan{\the\count1.0\the\count255}
  \else\xdef\klockan{\the\count1.\the\count255}\fi
\endgroup

\newcommand\nopf{\qed}   

\newcommand{\sumk}{\sum_{k=0}^\infty}
\newcommand{\sumki}{\sum_{k=1}^\infty}

\newcommand{\sumn}{\sum_{n=0}^\infty}

\newcommand{\sumin}{\sum_{i=1}^n}
\newcommand{\sumini}{\sum_{i=0}^{n-1}}

\newcommand{\sumkon}{\sum_{k=0}^n}
\newcommand{\prodin}{\prod_{i=1}^n}

\newcommand\set[1]{\ensuremath{\{#1\}}}

\newcommand\bigpar[1]{\bigl(#1\bigr)}
\newcommand\Bigpar[1]{\Bigl(#1\Bigr)}

\newcommand\lrpar[1]{\left(#1\right)}

\newcommand\xcpar[1]{\{#1\}}

\def\rompar(#1){\textup(#1\textup)}    

\newcommand\Bigparfrac[2]{\Bigpar{\frac{#1}{#2}}}

\def\xexp(#1){e^{#1}}

\newcommand\ntoo{\ensuremath{{n\to\infty}}}

\newcommand\downto{\searrow}

\newcommand\punkt[1]{\if.#1\else.\spacefactor1000\fi{#1}}

\newcommand\ie{i.e\punkt}
\newcommand\eg{e.g\punkt}
\newcommand\viz{viz\punkt}
\newcommand\cf{cf\punkt}

\newcommand{\tend}{\longrightarrow}
\newcommand\dto{\overset{\mathrm{d}}{\tend}}

\newcommand\eqd{\overset{\mathrm{d}}{=}}

\newcommand\bbR{\mathbb R}

\newcommand\bbZ{\mathbb Z}

\newcounter{CC}
\newcounter{cc}

\newcommand\E{\operatorname{\mathbb E{}}}
\renewcommand\P{\operatorname{\mathbb P{}}}
\newcommand\Var{\operatorname{Var}}
\newcommand\Cov{\operatorname{Cov}}

\newcommand\Be{\operatorname{Be}}

\newcommand\rise[1]{^{\overline{#1}}}

\newcommand\sign{\operatorname{sign}}

\newcommand\ga{\alpha}
\newcommand\gb{\beta}
\newcommand\gd{\delta}

\newcommand\gam{\gamma}
\newcommand\gG{\Gamma}

\renewcommand\phi{\xxx}  

\newcommand\cS{{\mathcal S}}

\newcommand\ett[1]{\boldsymbol1\xcpar{#1}}

\newcommand\qw{^{-1}}

\newcommand\qq{^{1/2}}
\newcommand\qqw{^{-1/2}}

\renewcommand{\=}{:=}

\newcommand\oi{[0,1]}

\newcommand\ddx{\mathrm{d}}

\newcommand{\pgf}{probability generating function}

\newcommand\rhs{right-hand side}

\newcommand\g{\gamma}
\renewcommand\b{\beta}
\renewcommand\a{\alpha} 
\newcommand\be{\begin{equation}}
\newcommand\ee{\end{equation}}
\newcommand\lbl{\label}

\newcommand\na{N_\ga}
\newcommand\nb{N_\gb}
\newcommand\nc{N_\gam}
\newcommand\nd{N_\gd}
\newcommand\st{staircase tableau}
\newcommand\gabst{$\ga/\gb$-staircase tableau}
\newcommand\css{\bar\cS}
\newcommand\cSx{\cS^*}
\newcommand\Zx{Z^*}
\newcommand\cSxx{\cS^{**}}
\newcommand\Zxx{Z^{**}}
\newcommand\cSxxi{\cS^{**{\prime}}}
\newcommand\tP{\tilde P}
\newcommand\tD{\tilde D}
\newcommand\tvx{\tilde v}
\newcommand\hD{\widehat D}
\newcommand\hP{\widehat P}

\newcommand\gab{_{\ga,\gb}}
\newcommand\ngab{_{n,\ga,\gb}}
\newcommand\noooo{_{n,\infty,\infty}}

\newcommand\ab{_{a,b}}
\newcommand\nab{_{n,a,b}}
\newcommand\niab{_{n-1,a,b}}

\newcommand\vx{v}
\newcommand\Pgab{\P_{\ga,\gb}}
\newcommand\sngab{S_{n,\ga,\gb}}
\newcommand\snoooo{S_{n,\infty,\infty}}
\newcommand\snoooox[1]{S_{n,\infty,\infty,#1}}
\newcommand\snoooorho{\snoooox{\rho}}
\newcommand\snooooq{\snoooox{1/2}}

\newcommand\pnab{P\nab}
\newcommand\pniab{P\niab}
\newcommand\vab{\vx\ab}

\newcommand\noo{_{n,0,0}}
\newcommand\tvxoo{\tvx_{0,0}}
\newcommand\tpnoo{\tP_{n,0,0}}

\newcommand{\euler}[2]{\genfrac{ < }{ > }{0pt}{}{#1}{#2}}

\newcommand\oivar{$0/1$-variable}
\newcommand\wt{\operatorname{wt}}
\newcommand\wtx{\operatorname{wt}}
\newcommand\hga{\hat\ga}
\newcommand\hgb{\hat\gb}
\newcommand\ha{\hat\a}
\newcommand\hb{\hat\b}
\newcommand\snj{\snx{j}}
\newcommand\snx[1]{\sxx{n}{#1}}
\newcommand\sxx[2]{S_{#1}(#2)}

\newcommand{\Polya}{P\'olya}

\newcommand\urladdrx[1]{{\urladdr{\def~{{\tiny$\sim$}}#1}}}

\begin{document}
\title
{Weighted random staircase tableaux}

\date{18 December, 2012}

\author[Pawe{\l} Hitczenko]{Pawe{\l} Hitczenko${}^\dagger$}
\thanks{$\dagger$ 
Department of Mathematics, Drexel University, Philadelphia, PA 19104, USA,
phitczenko@math.drexel.edu} 
\thanks{$\dagger$ Partially supported by 
Simons Foundation (grant \#208766 to Pawe{\l} Hitczenko)}
\address{Department of Mathematics, Drexel University, Philadelphia, 
PA  19104, USA} 
\email{phitczenko@math.drexel.edu}
\urladdrx{http://www.math.drexel.edu/~phitczen/}

\author[Svante Janson]{Svante Janson${}^\ddagger$}
\thanks{$\ddagger$ Department of Mathematics, Uppsala
University, Box 480, SE--751 06, Uppsala, Sweden, svante.janson@math.uu.se} 
\thanks{$\ddagger$
Partly supported by the Knut and Alice Wallenberg Foundation}
\address{Department of Mathematics, Uppsala University, PO Box 480,
SE-751~06 Uppsala, Sweden}
\email{svante.janson@math.uu.se}
\urladdrx{http://www2.math.uu.se/~svante/}



\begin{abstract} 
This paper concerns a relatively new combinatorial structure called
\st{x}. They were introduced in the context of the
asymmetric exclusion process and Askey--Wilson polynomials, however,
their purely 
combinatorial properties have gained considerable interest in the past
few years.

In this paper we further study  combinatorial properties of
\st{x}. We consider a general model of \st{x} in which symbols
that appear in \st{x} may have arbitrary positive weights. Under
this general model we derive  a number of results. Some of
our results concern the limiting laws for the number of appearances of
symbols in a random \st{x}.  They generalize and subsume earlier results
that were obtained 
for specific values  of the weights.  

One advantage of
our generality is that we may let the weights approach extreme values of
zero or infinity which covers further special cases appearing
earlier in the literature. Furthermore, our generality allows us to
analyze the structure of random \st{x} and we 
obtain several  results in this direction. 

One of the tools we use are generating functions of the parameters of
interests. This leads us to a two--parameter family of polynomials and
we study this family as well. Specific values of the parameters
include number of special cases analyzed earlier in the literature.  All of
them are generalizations of the classical Eulerian 
polynomials. 

We also briefly discuss the relation of \st{x} to the asymmetric
exclusion process, to other recently introduced types of tableaux, and
to an urn model studied by a number of researchers, including Philippe
Flajolet.  
\end{abstract}

\maketitle

\section{Introduction and main results}\label{S:intro}

This paper is concerned with a combinatorial structure introduced 
recently by Corteel and Williams \cite{cw_pnas, cw}  and called 
\emph{staircase tableaux}.   
The original motivations were  in 
connections with the asymmetric exclusion process (ASEP) on a
one-dimensional lattice with open boundaries,  an important  model in
statistical mechanics,
see \refS{SASEP} below for a brief summary and \cite{cw} for the full story.
The generating function for staircase tableaux was also used to 
give a combinatorial formula for the
moments of the 
Askey--Wilson polynomials (see \cite{cw,cssw} for the details). 
Further work includes \cite{cd-h} 
where special situations in which the generating function of staircase
tableaux took a particularly simple form were considered. Furthermore,
\cite{d-hh} deals with the analysis  
of various parameters associated with appearances of the Greek  letters $\a$,
$\b$, $\delta$, and $\g$ in a randomly chosen staircase tableau
(see below,  or \eg{} \cite[Section~2]{cw},  for the definitions and  the
meaning of these symbols). 
Moreover, there are natural bijections (see \cite[Appendix]{cw}) between
the a class of \st{x} (the \emph{\gabst{x}} defined below) and
\emph{permutation tableaux}
(see \eg{} \cite{ch,CW1,CW2,hj} and the references therein for more information
on these objects and their connection to a version of the ASEP) as well as
to 
\emph{alternative tableaux} \cite{n} which, in turn, are in one-to-one
correspondence with  
\emph{tree-like tableaux} \cite{abn}; we discuss this further
in \refS{Sperm}.

We recall the definition of a staircase tableau
 introduced in \cite{cw_pnas,cw}: 
\begin{definition}
A {\it staircase tableau of size $n$} is a Young diagram of  shape
$(n,n-1,\dots,2,1)$ whose boxes are filled according to the following rules: 
\begin{stairenumerate}[5pt]
\item\label{st0} 
each box is either empty or contains one of the letters $\a$, $\b$,
$\delta$, or $\g$; 
\item\label{stdiag} 
no box on the diagonal is empty;
\item\label{stbd} 
all boxes in the same row and to the left of a $\b$ or a $\delta$ are empty;
\item\label{stag} 
all boxes in the same column and above an $\a$ or a $\g$ are empty.
\end{stairenumerate}
\end{definition}
An example of a staircase tableau  is given in \refF{fig:tab}.


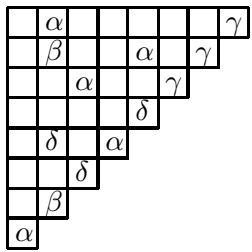
\begin{figure}[htbp]
\setlength{\unitlength}{0.4cm}
\begin{picture}(10,8)(0,0)\thicklines  
\put(0,0){\line(0,0){8}}
\put(1,0){\line(0,1){8}}
\put(2,1){\line(0,1){7}}
\put(3,2){\line(0,1){6}}
\put(4,3){\line(0,1){5}}
\put(5,4){\line(0,1){4}}
\put(6,5){\line(0,1){3}}
\put(7,6){\line(0,1){2}}
\put(8,7){\line(0,1){1}}

\put(0,8){\line(1,0){8}}
\put(0,7){\line(1,0){8}}
\put(0,6){\line(1,0){7}}
\put(0,5){\line(1,0){6}}
\put(0,4){\line(1,0){5}}
\put(0,3){\line(1,0){4}}
\put(0,2){\line(1,0){3}}
\put(0,1){\line(1,0){2}}
\put(0,0){\line(1,0){1}}

\put(1.25,7.25){$\alpha$}
\put(7.25,7.25){$\gamma$}
\put(0.25,0.25){$\alpha$}
\put(1.25,1.25){$\beta$}
\put(2.25,2.25){$\delta$}
\put(3.25,3.25){$\alpha$}
\put(4.25,4.25){$\delta$}
\put(5.25,5.25){$\gamma$}
\put(6.25,6.25){$\gamma$}
\put(1.25,3.25){$\delta$}
\put(2.25,5.25){$\alpha$}
\put(4.25,6.25){$\alpha$}
\put(1.25,6.25){$\beta$}

\end{picture}
\caption{ 
A staircase tableau of size 8; its weight is 
$\alpha^5\beta^2\delta^3\gamma^3$.}
\lbl{fig:tab}
\end{figure}

The set of all staircase tableaux of size $n$ will be denoted by $\cS_n$.
There are several proofs of the fact that the number of staircase tableaux 
$|\cS_n|=4^nn!$,
see \eg{} \cite{cssw,cd-h,d-hh} for some of them,
or \eqref{sim_gen} below and its proof in \refS{Spf}.

Given a staircase tableau $S$, we let $\na,\nb,\nc,\nd$ be the
numbers of 
symbols $\ga,\gb,\gam,\gd$ in $S$.
(We also use the notation $\na(S)$, \dots) 
Define the \emph{weight} of $S$ to be
\begin{equation}\label{w}
  \wt(S)\=\ga^{\na}\gb^{\nb}\gam^{\nc}\gd^{\nd},
\end{equation}
\ie, the product of all symbols in $S$.
(This is a simplified version;
see \refS{SASEP} for the more general version including further variables $u$
and $q$. This is
used \eg{} in the connection with the ASEP \cite{cw},
see \cite{cssw} for further properties, 
but we will in this paper only consider the version above, which is
equivalent to taking $u=q=1$.)

By \ref{stag}, each column contains at most one $\ga$ or $\gam$, and thus
$\na+\nc\le n$. 
Similarly, by \ref{stbd}, each row contains at most one $\gb$ or $\gd$ so
$\nb+\nd\le n$. Together with \ref{stdiag} this yields
\begin{equation}
  n \le \na+\nb+\nc+\nd \le 2n.
\end{equation}
Actually, as is seen from \eqref{sim_gen} below, the maximum
of $\na+\nb+\nc+\nd$ is $2n-1$, see also \refE{Eoooo} and \refS{Smax}.
Note that there are $n(n+1)/2$ boxes in a staircase tableau in $\cS_n$.
Hence, in a large staircase tableau, only a small proportion of the boxes
are filled.

The generating function 
\begin{equation}
Z_n(\ga,\gb,\gam,\gd)\=\sum_{S\in\cS_n} \wt(S)  
\end{equation}
has a particularly simple form, \viz, 
see \cite{cssw,cd-h},
\begin{equation}
\lbl{sim_gen}
Z_n(\a,\b,\delta,\g)
=\prod_{i=0}^{n-1}\bigpar{\a+\b+\delta+\g+i(\a+\g)(\b+\delta)}.
\end{equation}
(A proof is included in \refS{Spf}.)
In particular, the number of staircase tableaux of size $n$ is
$Z_n(1,1,1,1)=\prod_{i=0}^{n-1}(4+4i)=4^nn!$, as said above.
(We use $\ga,\gb,\gam,\gd$ as fixed symbols in the tableaux, 
and in $\na,\dots,\nd$,
but otherwise as
variables or real-valued parameters.
This should not cause any confusion.)

Note that the symbols
$\ga$ and $\gam$ have exactly the same role in the definition
above of staircase tableaux, and so do $\gb$ and $\gd$. (This is no longer
true in the connection to the ASEP, see \refS{SASEP},
which is the reason for using four different
symbols in the definition.) We say that a staircase tableau using only the
symbols $\ga$ and 
$\gb$ is an \emph{\gabst}, and we let $\css_n\subset\cS_n$ be the set of all
\gabst{x} of size $n$. 
We thus see that 
any staircase tableau can be 
obtained from an \gabst{} by replacing some (or no) $\ga$ by $\gam$ and
some (or no) $\gb$ by $\gd$; conversely, 
any staircase tableau can be reduced to an \gabst{} by replacing
every $\gam$ by $\ga$ and every $\gd$ by $\gb$.

We define the generating function of \gabst{x} by
\begin{equation}\label{zab0}
Z_n(\ga,\gb)\=\sum_{S\in \css_n} \wt(S)
=Z_n(\ga,\gb,0,0),  
\end{equation}
and note that the relabelling argument just given implies
\begin{equation}\label{ja}
Z_n(\a,\b,\gam,\gd)=Z_n(\a+\g,\b+\delta).  
\end{equation}
We let $x\rise n$  denote the rising factorial defined by
\begin{equation}
  x\rise n\=x(x+1)\dotsm(x+n-1)=\gG(x+n)/\gG(x),
\end{equation}
and note that 
 by \eqref{sim_gen},
\begin{equation}
\label{zab}
\begin{split}
Z_n(\ga,\gb)&=Z_n(\ga,\gb,0,0)
=\prod_{i=0}^{n-1}(\a+\b+i\a\b)
=\ga^n\gb^n(\ga\qw+\gb\qw)\rise n
\\&
=\ga^n\gb^n\frac{\gG(n+\ga\qw+\gb\qw)}{\gG(\ga\qw+\gb\qw)}.
\end{split}
\end{equation}

In particular, as noted in \cite{cd-h} and \cite{cssw}, the number of
\gabst{x} is 
$Z_n(1,1)=2\rise n=(n+1)!$.

\citet{d-hh} studied random staircase tableaux obtained by picking a
staircase tableau in $\cS_n$ uniformly at random. We can obtain the same
result by picking an \gabst{} in $\css_n$ at random with probability
proportional to $2^{\na+\nb}$ and then randomly replacing some symbols;
each $\ga$ is replaced by $\gam$
with probability $1/2$, and each $\gb$ by $\gd$  with probability $1/2$,
with all replacements independent.
Note that the weight $2^{\na+\nb}$ is the weight \eqref{w} if we choose the
parameters $\ga=\gb=2$.
The purpose of this paper is to, more generally, study random \gabst{x}
defined similarly with weights of this type for arbitrary parameters 
$\ga,\gb\ge0$.
(As we will see in \refS{Sex}, this includes several cases considered
earlier. 
It will also be useful in studying the structure of random \st{x}, see
\refS{Ssub}.) 
We generalize several results from \cite{d-hh}.

\begin{definition}\label{Dsngab}
  Let $n\ge1$ and let $\ga,\gb\in[0,\infty)$ with $(\ga,\gb)\neq(0,0)$.
Then $\sngab$ is the random \gabst{} in $\css_n$ with the distribution
\begin{equation}\label{pab}
\Pgab(\sngab=S)
=\frac{\wt(S)}{Z_n(\ga,\gb)}=\frac{\ga^{\na(S)}\gb^{\nb(S)}}{Z_n(\ga,\gb)},
\qquad S\in \css_n.
\end{equation}
We also allow the parameters $\ga=\infty$ or $\gb=\infty$;
in this case \eqref{pab} is interpreted as the limit when $\ga\to\infty$ or
$\gb\to\infty$, with the other parameter fixed. Similarly, we allow
$\ga=\gb=\infty$; in this case 
\eqref{pab} is interpreted as the limit when $\ga=\gb\to\infty$.
See further Examples \ref{Eoo}--\ref{Eoooo} and \refS{Smax}.
(In the case $\ga=\gb=\infty$, we tacitly assume $n\ge2$ or sometimes
even $n\ge3$  to avoid trivial complications.)
\end{definition}

\begin{remark}\label{Rsymm}
  There is a symmetry (involution) $S\mapsto S^\dagger$ 
of   staircase tableaux
defined by reflection in the NW--SE diagonal, thus interchanging rows and
  columns, together with an  exchange of the symbols by
  $\ga\leftrightarrow\gb$ and 
$\gam\leftrightarrow\gd$, see further \cite{cd-h}.
This maps $\css_n$ onto itself, and maps the random
\gabst{} $\sngab$ to $S_{n,\gb,\ga}$; the parameters $\ga$ and $\gb$ thus
play symmetric roles.
\end{remark}

\begin{remark}\label{R4}
  We can similarly define a random staircase tableaux
  $S_{n,\ga,\gb,\gam,\gd}$, with four parameters $\ga,\gb,\gam,\gd\ge0$, by
  picking a staircase tableau $S\in \cS_n$ with probability
  $\wt(S)/Z_n(\ga,\gb,\gam,\gd)$. 
By the argument above, this is the same as
taking a random $S_{n,\ga+\gam,\gb+\gd}$ and randomly replacing each symbol
$\ga$ by $\gam$ with probability $\gam/(\ga+\gam)$,
and each $\gb$ by $\gd$ with probability $\gd/(\gb+\gd)$.
(The case $\ga=\gb=\gam=\gd=1$ was mentioned above.)
Our results can thus be translated to 
results for   $S_{n,\ga,\gb,\gam,\gd}$, 
but we leave this to the reader. 
\end{remark}

\begin{remark}\label{R0}
  For convenience (as a base case in inductions) we allow also $n=0$;
$S_0=\css_0$ contains a single, empty staircase tableaux with
  $\na=\nb=\nc=\nd=0$ and thus weight $\wt=1$, so $Z_0=1$.
(At some places, \eg{} in \refS{Smax}, we assume $n\ge1$ to avoid
  trivial complications.)
\end{remark}

\begin{remark}
  It seems natural to use the parameters $\ga$ and $\gb$ as above in
  \refD{Dsngab}. However, in many of our results it is more convenient,
and sometimes perhaps more natural, to
  use $\ga\qw$ and $\gb\qw$ instead.
We will generally use the notations $a\=\ga\qw$  and $b\=\gb\qw$, and formulate
results in terms of these parameters whenever convenient.
\end{remark}

We are interested in the distribution of various parameters
of $\sngab$.
In particular, we define $A(S)$ and $B(S)$ as the numbers of $\ga$ and
$\gb$, respectively, on the diagonal of an \gabst{} $S$, and consider the
random variables
$A\ngab\=A(\sngab)$ and $B\ngab\=B(\sngab)$;
note that $A\ngab+B\ngab=n$ by \ref{stdiag},
so it suffices to consider one of these.
Moreover, by \refR{Rsymm}, 
$B\ngab\eqd A_{n,\gb,\ga}$.

In order to describe the distribution of $A\ngab$ we need some further
notation.  Define
numbers $\vx\ab(n,k)$, 
for $a,b\in\bbR$, $k\in\bbZ$ and $n=0,1,\dots$,
by the recursion
\begin{equation}\label{vx-rec}
\vab(n,k)=(k+a)\vab(n-1,k)+(n-k+b)\vab(n-1,k-1),
\qquad n\ge1,
\end{equation}
with $\vab(0,0)=1$ and $\vab(0,k)=0$ for $k\neq0$, see \refTab{Tab:vab}.
(It is convenient to define $\vab(n,k)$  for all integers
$k\in\bbZ$, but note that $\vab(n,k)=0$ for $k<0$ and $k>n$, for all
$n\ge0$,
so it really suffices to consider $0\le k\le n$.)
Furthermore, define
polynomials 
\begin{equation}\label{pnab}
  \pnab(x)\=\sumkon \vx\ab(n,k)x^k
=\sum_{k=-\infty}^\infty \vx\ab(n,k)x^k.
\end{equation}
Thus, $P_{0,a,b}(x)=1$. Moreover, the recursion \eqref{vx-rec} is easily seen
to be equivalent to the recursion
\begin{equation}\label{pnab-rec}
  P_{n,a,b}(x)
= \bigpar{(n-1+b)x+a}  P_{n-1,a,b}(x)
+x(1-x)   P_{n-1,a,b}'(x),
\qquad n\ge1.
\end{equation}

\begin{table}
  \begin{tabular}{l|c|c|c|c|}
	$n\backslash k$ & 0 & 1 & 2 & 3 
\\ \hline
0 & 1 & &&
\\ \hline
1 & $a$ & $b$& &
\\ \hline
2 & $a^2$& $a+b+2ab$ & $b^2$ & 
\\ \hline
3 & $a^3$& $a+b+3a^2+3ab+3a^2b$& $a+b+3ab+3b^2+3ab^2$& $b^3$
\\ \hline
  \end{tabular}
\vskip 4pt 
\caption{The coefficients $\vab(n,k)$ of $\pnab$ for small $n$.}
\label{Tab:vab}
\end{table}

In the cases $(a,b)=(1,0)$, $(0,1)$ and $(1,1)$, the numbers $\vab(n,k)$
are the Eulerian numbers and $\pnab(x)$ are the Eulerian polynomials (in
different versions), see \refS{SEuler}.
We can thus see $\vab(n,k)$ and $\pnab(x)$ as generalizations of Eulerian
numbers and polynomials.
Some properties of these numbers and polynomials are given in \refS{Spol},
where we also discuss  some other cases that have been considered earlier.

In the case $a=b=0$, we trivially have $\vx_{0,0}(n,k)=0$ and $P_{n,0,0}=0$
for all $n\ge1$; in this case we define the substitutes, for $n\ge2$,
\begin{equation}\label{tvx00}
  \tvx_{0,0}(n,k)\=v_{1,1}(n-2,k-1)
\end{equation}
and 
\begin{equation}\label{tpn00}
\tP_{n,0,0}(x)\=\sumkon  \tvx_{0,0}(n,k) x^k = xP_{n-2,1,1}(x).
\end{equation}
See further Lemmas \ref{L00rec} and \ref{L00lim}. 

Our main results are the following.
Proofs are given in \refS{Spf}. 

\begin{theorem}\label{T1}
Let $\ga,\gb\in(0,\infty]$ and let $a\=\ga\qw$, $b\=\gb\qw$.
If $(\ga,\gb)\neq(\infty,\infty)$, then
the \pgf{} $g_A(x)$ of the random variable $A\ngab$ is given by
\begin{equation}\label{t1}
  \begin{split}
  g_A(x)&\=
\E x^{A\ngab}=\sumkon \P(A\ngab=k)x^k 
\\&\phantom:
= \frac{\pnab(x)}{\pnab(1)}
= \frac{\pnab(x)}{(a+b)\rise n}
= \frac{\gG(a+b)}{\gG(n+a+b)} \pnab(x).	
  \end{split}
\end{equation}
Equivalently,
\begin{equation}\label{t1b}
  \begin{split}
\P(A\ngab=k)
= \frac{\vx\ab(n,k)}{\pnab(1)}
= \frac{\vx\ab(n,k)}{(a+b)\rise n}
= \frac{\gG(a+b)}{\gG(n+a+b)} \vx\ab(n,k).	
  \end{split}
\end{equation}

In the case $\ga=\gb=\infty$, 
and $n\ge2$, 
we have instead
\begin{gather}
  g_A(x)\=
\sumkon \P(A\ngab=k)x^k 
= \frac{\tpnoo(x)}{\tpnoo(1)}
= \frac{\tpnoo(x)}{(n-1)!},
\\  \label{t100p}
\P(A\ngab=k)
= \frac{\tvxoo(n,k)}{\tpnoo(1)}
= \frac{\tvxoo(n,k)}{(n-1)!}  .
\end{gather}
\end{theorem}

\begin{theorem} \label{T2}
Let $\ga,\gb\in(0,\infty]$ and let $a\=\ga\qw$ and $b\=\gb\qw$. Then
\begin{equation*}
  \begin{split}
\E(A\ngab)
=
\frac{n(n+2b-1)}
{2(n+a+b-1)}
  \end{split}
\end{equation*}
and
\begin{equation*}
  \begin{split}
&\Var(A\ngab)
\\&\quad
=
n\frac{(n-1)(n-2)(n+4a+4b-1)
+6(n-1)(a+b)^2
+12ab(a+b-1)}
{12(n+a+b-1)^2(n+a+b-2)}	
.
  \end{split}
\end{equation*}
\end{theorem}

\begin{remark}\label{RT2}
  In the symmetric case $\ga=\gb$ we thus obtain $\E (A_{n,\ga,\ga})=n/2$; 
this is also obvious by symmetry, since $A_{n,\ga,\ga}\eqd B_{n,\ga,\ga}$
by \refR{Rsymm}.
\end{remark}

\begin{theorem}\label{Tneg}
  The \pgf{} $g_A(x)$ of the random variable $A\ngab$ has all its roots
  simple and on the negative halfline $(-\infty,0]$.
As a consequence, for any given $n,\ga,\gb$
there exist $p_1,\dots,p_n\in\oi$ 
such that
\begin{equation}\label{tneg}
  A\ngab\eqd \sumin \Be(p_i),
\end{equation}
where $\Be(p_i)$ is a Bernoulli random variable with parameter $p_i$ 
and the  summands are independent.
It follows that the distribution of $A\ngab$ and the sequence
$\vx\ab(n,k)$, $k\in\bbZ$, are
unimodal and log-concave.
\end{theorem}

These results lead to a central limit theorem: 
\begin{theorem}\label{TCLT}
Let $\ga,\gb\in(0,\infty]$ be fixed and let \ntoo. Then $A\ngab$ is
  asymtotically normal:
  \begin{equation}\label{tclt1}
\frac{A\ngab-\E A\ngab}{(\Var A\ngab)\qq}\dto N(0,1),
  \end{equation}
or, more explicitly,
  \begin{equation}\label{tclt2}
\frac{A\ngab-n/2}{\sqrt n}\dto N(0,1/12).	
  \end{equation}
Moreover, a corresponding local limit theorem holds: 
\begin{equation}\label{tclt3a}
  \P(A\ngab=k)
=\bigpar{2\pi\Var A\ngab}\qqw
\Bigpar{e^{-\frac{(k-\E A\ngab)^2}{2\Var A\ngab}}+o(1)}, 
\end{equation}
as \ntoo, uniformly in $k\in\bbZ$, or, more explicitly,
\begin{equation}\label{tclt3b}
  \P(A\ngab=k)=\sqrt{\frac6{\pi n}}\Bigpar{e^{-6(k-n/2)^2/n}+o(1)},
\end{equation}
as \ntoo, uniformly in $k\in\bbZ$.
\end{theorem}

\begin{remark}
\label{RCLT} 
The proof shows that the central limit theorem in the forms \eqref{tclt1}
and \eqref{tclt3a} holds also if $\ga$ and $\gb$ are allowed to
depend on $n$, provided only that $\Var(A\ngab)\to\infty$, which by
\refT{T2} holds as soon as 
$n^2/(a+b)\to\infty$ or $nab/(a+b)^2\to\infty$;
hence this holds except when
$a$ or $b$ is $\infty$ or tends  to $\infty$ rapidly,
\ie, unless $\ga$ or $\gb$ is 0 or tends to 0 rapidly.
\refE{E0} illustrates that asymptotic normality may fail in extreme cases.
\end{remark}

We can also study the total numbers $\na$ and $\nb$ of symbols $\ga$ and
$\gb$ in a random $\sngab$.
This is simpler, and follows directly from \eqref{zab}, as we show in
\refS{Spf}.
(Recall that in $\na$ and $\nb$, $\ga$ and $\gb$ are symbols and not
parameter values.)

\begin{theorem}
  \label{TN1}
Let $\ga,\gb\in(0,\infty]$, and let $a\=\ga\qw$, $b\=\gb\qw$.
The joint \pgf{} of $\na$ and $\nb$ for the random \st{} $\sngab$ is
\begin{equation}\label{tn1}
  \E\gab\lrpar{x^{\na} y^{\nb}}
=
\prod_{i=0}^{n-1}\frac{\ga x+\gb y+i\ga\gb xy}{\ga+\gb+i\ga\gb}
=
\prod_{i=0}^{n-1}\frac{b x+a y+i xy}{a+b+i}.
\end{equation}
In other words,
\begin{equation}
  \label{tn1x}
  \bigpar{\na,\nb}\eqd\lrpar{\sumini I_i, \sumini J_i},
\end{equation}
where $(I_i,J_i)$ are independent pairs of \oivar{s} with the
distributions
\begin{equation}\label{tn1b}
  \P(I_i=\iota,J_i=\iota')
=
\begin{cases}
  0, & (\iota,\iota')=(0,0), \\
  \frac{b}{a+b+i}, & (\iota,\iota')=(1,0), \\
  \frac{a}{a+b+i}, & (\iota,\iota')=(0,1), \\
  \frac{i}{a+b+i}, & (\iota,\iota')=(1,1). 
\end{cases}
\end{equation}
In particular, the marginal distributions are 
\begin{align}\label{tn1c}
 I_i\sim\Be\Bigpar{1-\frac{a}{a+b+i}},&&&J_i\sim\Be\Bigpar{1-\frac{b}{a+b+i}}. 
\end{align}
Hence,
\begin{align}
  \E \na &= \sumini \Bigpar{1-\frac{a}{a+b+i}}
= n - \sumini \frac{a}{a+b+i}, \label{tn1e}\\
  \Var \na &= \sumini \frac{a}{a+b+i}\Bigpar{1-\frac{a}{a+b+i}}
\label{tn1var}, \\
\Cov(\na,\nb)& = - \sumini \frac{ab}{(a+b+i)^2}. \label{tn1cov}
\end{align}

In the case $\ga=\gb=\infty$ ($a=b=0$) and $i=0$, we interpret
$\frac{a}{a+b+i}=\frac{b}{a+b+i}=\frac12$
and $\frac{i}{a+b+i}=0$ in \eqref{tn1b}--\eqref{tn1cov},
and
the factor in \eqref{tn1} as $(x+y)/2$.
\end{theorem}

\begin{theorem}
  \label{TN2}
Let $\ga,\gb\in(0,\infty]$ be fixed and let \ntoo.
Then, with $a\=\ga\qw$ and $b\=\gb\qw$, 
\begin{align}
  \E \na &= n - a \log n + O(1) \label{tn2e}\\
  \Var \na &= a\log n + O(1), \label{tn2var}\\ 
\Cov(\na,\nb)& = O(1). \label{tn2cov}
\end{align}
Furthermore,
\begin{align}
  \frac{\na-\E\na}{\sqrt{\log n}}&\dto N(0,a), \label{tn2a}\\
  \frac{\nb-\E\nb}{\sqrt {\log n}}&\dto N(0,b), \label{tn2b}
\end{align}
jointly, with independent limits.
\end{theorem}

\begin{remark}
A local limit theorem holds too. 
Moreover, \refT{TN1} implies that $n-\na$ can be approximated 
in total variation sense by a Poisson distribution  $\P(n-\E\na)$,
see \eg{} \cite[Theorem 2.M]{SJI}. We omit the details.
\end{remark}

\begin{remark}
  We can similarly also study the joint distribution of, \eg, $\na$ and $A$
  (the total number of $\ga$'s and the number on the diagonal), but we leave
  this to the reader.
\end{remark}

The results above show that the effects of changing the parameters $\ga$ and
$\gb$ are surprisingly small. Typically, probability weights of the type
\eqref{w} (which are common in statistical physics) shift the distributions
of the random variables considerably, but here the effects in \eg{} Theorems
\refand{T2}{TN2} are only second-order. The reason seems to be that the
variables are so constrained; we have $\na,\nb\le n$ and by \refT{TN1}, both
are close to their maximum and thus the weights do not differ as much
between different random \st{x} as might be expected.

\begin{remark}
  In order to get stronger effects, we may let the weights tend to 0 as
  \ntoo. For example, taking $\ga=1/(sn)$ and $\gb=1/(tn)$ for some fixed
  $s,t>0$, and thus $a=sn$, $b=tn$, we obtain by \refT{T2}
  \begin{align}
\E(A\ngab)
&=
\frac{2t+1}{2(s+t+1)}n+O(1),
\\
\Var(A\ngab)
&
=
\frac{1+4s+4t
+6(s+t)^2
+12st(s+t)}
{12(s+t+1)^3}	
\,n+O(1).	
  \end{align}
A central limit theorem holds by \refR{RCLT}. 
Similarly, one easily shows joint asymptotic normality for $\na,\nb$ in this
case too; unlike the case of fixed $\ga$ and $\gb$ in \refT{TN2}, the limits
are now dependent normal variables. We omit the details.
 Note that by \refT{Tsub}, the central part of a uniformly random \gabst,
say the part comprising the middle third of the rows and columns, is an
example of this type.
\end{remark}

We discuss some examples in \refS{Sex}. 
Sections \ref{SEuler}--\ref{Spol} contain further preliminaries, and the
proofs of the theorems above are given in \refS{Spf}.
Sections \ref{Ssub} and~\ref{Sloc} contain further results on subtableaux
and on the positions of the symbols in a random \st.
The limiting case $\ga=\gb=\infty$ is studied in greater detail in \refS{Smax}. 
\refS{Surn} discusses an urn model which gives the same distribution as
$A\ngab$. 
\refS{Sperm} discusses, as said above, some other, equivalent,
types of tableaux. \refS{SASEP}, finally, describes briefly the connection
to ASEP mentioned above.

\section{Special cases}\label{Sex}

\begin{example}\label{E22}
$\ga=\gb=2$. As said above, this yields 
the uniformly random staircase
tableaux studied by \citet{d-hh}.
More precisely, in the notation of \refR{R4}, the uniformly random staircase
tableaux is $S_{n,1,1,1,1}$, which is obtained from $S_{n,2,2}$ by a simple
random replacement of symbols.

The main results of \cite{d-hh} can be recovered as special cases of the
theorems above, with $a=b=1/2$.
Note that in this case, the formulas in \refT{T2} simplify to 
$\E(A_{n,2,2})=n/2$ (see \refR{RT2})
and $\Var(A_{n,2,2})=(n+1)/12$.

Recall that the number of all \st{x} of size $n$ is 
$Z_n(1,1,1,1)=Z_n(2,2)=4^n n!$,
see \eqref{sim_gen} and \eqref{zab}.
\end{example}

\begin{example}\label{E11}
  $\ga=\gb=1$. This yields the uniformly random \gabst{} $S_{n,1,1}$.
As said above, the number of \gabst{x} of size $n$ is $Z_n(1,1)=(n+1)!$.
Indeed, \citet{cw} gave a bijection between \gabst{x} of size $n$
and permutation tableaux of size (length) $n+1$, and there are several
bijections 
between the latter and permutations of size $n+1$ \cite{SW,CN};
see further \refS{Sperm}. 
\gabst{x} are further studied in \cite{cssw,cd-h}.

The theorems above, with $a=b=1$, yield results on uniformly random \gabst{x}.
For example, \refT{T1} shows, using \eqref{v11}, that the distribution of
$A_{n,1,1}$ is given by the Eulerian numbers:
\begin{equation}\label{euler11}
  \P(A_{n,1,1}=k)=\frac{\vx_{1,1}(n,k)}{(n+1)!}
=\frac{\euler{n+1}k}{(n+1)!}.
\end{equation}
In other words, the number of \gabst{x} of size $n$ with $k$ $\ga$'s on the
diagonal is $\euler{n+1}k$. 
(This follows also by the bijections mentioned above between
\gabst{x} and permutation tableaux \cite{cw} and between
the latter and permutations \cite{CN}.)
Theorems \refand{Tneg}{TCLT} give in this case well-known results for
Eulerian numbers, see
\cite{Frobenius} and \cite{Carlitz-asN}, respectively.

Furthermore, the formulas in \refT{T2} simplify and yield
$\E A_{n,1,1}=n/2$ (see \refR{RT2}) and $\Var A_{n,1,1}=(n+2)/12$.
As another example, 
\refT{TN1} shows that
\begin{equation}\label{e11}
  n-\na\eqd \sumini (1-I_i) \sim 
\sum_{i=0}^{n-1} \Be\Bigparfrac{1}{i+2}
= \sum_{i=2}^{n+1} \Be\Bigparfrac{1}{i},
\end{equation}
with the summands independent; note that this has the same distribution as
$C_{n+1}-1$, where $C_{n+1}$ is the number of cycles in a random permutation of
size $n+1$, or, equivalently, the number of maxima (records) in such a
random permutation. (Again, a bijective proof can be given using the
bijections with permutation tableaux and permutations
in  \cite{cw} and \cite{CN}.)
See also \refS{Sperm}.
\end{example}

\begin{example}\label{E21}
$\a=2$, $\b=1$ corresponds to staircase tableaux without
$\delta$'s briefly studied in \cite{cssw}. 
The number of such \st{x} is, by \eqref{zab},
\begin{equation}
 Z_n(2,1)=2^n(3/2)\rise{n} = \prod_{i=0}^{n-1}(3+2i)=(2n+1)!!,
\end{equation}
 see \cite{cssw,cd-h}. 
Our theorems yield results on random $\gd$-free \st{x}.
\end{example}

\begin{example}\label{Eoo}
  $\ga=\infty$.
This means that we take the limit as $\ga\to\infty$ in \eqref{pab}, which
means that we have a non-zero probability only for staircase
tableaux with the maximum number of symbols $\ga$, \ie, with $\na=n$.
For such \gabst{x}, the probability is proportional to $\gb^{\nb}$.

We let $\cSx_n\subset \css_n$ be the set of such \gabst{} of size $n$;
by \ref{stag}, these are the \gabst{} of size $n$
with exactly one $\ga$ in each column.
(Such staircase tableaux were studied in
\cite{cd-h}.)
We define the corresponding generating function
\begin{equation}\label{zxn}
  \Zx_n(\gb)\=\sum_{S\in\cSx_n} \gb^{\nb} 
=  \lim_{\ga\to\infty} \ga^{-n} Z_n(\ga,\gb)
=\prod_{i=0}^{n-1}(1+i\b),
\end{equation}
where the final equality follows from \eqref{zab}.
Thus, $S_{n,\infty,\gb}$ is the random element of $\cSx_n$ with the
distribution 
$\P(S_{n,\infty,\gb}=S)=\gb^{\nb(S)}/\Zx_n(\gb)$. 
\end{example}

\begin{example}\label{Eoo1}
$\ga=\infty$, $\gb=1$.
As a special case of the preceding example, $S_{n,\infty,1}$ is a uniformly
random element of $\cSx_n$. 
By \eqref{zxn}, the number of \gabst{x} of size $n$ with $n$ $\ga$'s is
\begin{equation}
  \Zx_n(1)=n!.
\end{equation}
Hence, the probability that a uniformly random \gabst{} has the maximum
number $n$ of $\ga$'s is $\Zx_n(1)/Z_n(1,1)=n!/(n+1)!=1/(n+1)$.
(See also \refT{TN1} and \eqref{e11}.)

The theorems above, with $a=0$ and $b=1$, yield results on uniformly random
\gabst{x} with $n$ $\ga$'s (\ie, one in each column).
For example, \refT{T1} shows, using \eqref{esymm}, that the distribution of
$A_{n,\infty,1}$ is given by the Eulerian numbers:
\begin{equation}\label{euleroo1}
  \P(A_{n,\infty,1}=k)=\frac{\vx_{0,1}(n,k)}{n!}
=\frac{\euler{n}{k-1}}{n!}.
\end{equation}
In other words, the number of \gabst{x} of size $n$ with $n$ $\ga$'s of which
$k$ are on the diagonal is $\euler{n}{k-1}$. 
(A bijective proof is given in \cite{cd-h}.)
By symmetry, counting instead the number of $\gb$'s on the diagonal, by
\eqref{v10}, 
\begin{equation}\label{euleroob}
  \P(B_{n,\infty,1}=k)=
  \P(A_{n,1,\infty}=k)=\frac{\vx_{1,0}(n,k)}{n!}
=\frac{\euler{n}{k}}{n!}.
\end{equation}
Compare with \refE{E11}, where also the distributions of $A$ and $B\eqd A$ are
given by Eulerian numbers. We see that by \eqref{euleroo1}--\eqref{euleroob} and
\eqref{euler11} that 
$A_{n,\infty,1}\eqd A_{n-1,1,1}+1$
and 
$B_{n,\infty,1}\eqd A_{n,\infty,1}-1\eqd A_{n-1,1,1}\eqd B_{n-1,1,1}$.

The formulas in \refT{T2} simplify and yield
$\E A_{n,\infty,1}=(n+1)/2$ and $\Var A_{n,\infty,1}=(n+1)/12$.
As another example, 
\refT{TN1} shows that
\begin{equation}
  n-\nb\eqd \sumini (1-J_i) \sim 
\sum_{i=0}^{n-1} \Be\Bigparfrac{1}{i+1}
= \sum_{i=1}^{n} \Be\Bigparfrac{1}{i},
\end{equation}
with the summands independent;  this has the same distribution as
$C_{n}$, with $C_{n}$ as in the corresponding result in \refE{E11}.
(A bijective proof is given in \cite{cd-h}.)
\end{example}

\begin{example}
  \label{Eoooo}
$\ga=\gb=\infty$. This means that we take the limit as $\ga=\gb\to\infty$ in 
\eqref{pab}, which means that we have a non-zero probability only for
\gabst{} with the maximum number of symbols.
These tableaux correspond to the terms with maximal total degree in
$Z_n(\ga,\gb)$, and it follows from \eqref{zab} that they have $2n-1$
symbols. (We assume $n\ge1$.)

We let $\cSxx_n\subset \css_n $ be the set of \gabst{} with $\na+\nb=2n-1$;
thus $S_{n,\infty,\infty}$ is a uniformly random element of $\Zxx$.

We further define the corresponding generating function
\begin{equation}
  \Zxx_n(\ga,\gb)\=\sum_{S\in\cSxx_n}\ga^{\na} \gb^{\nb} .
\end{equation}
This can be obtained by extracting the terms with largest degrees in
\eqref{zab}, and thus
\begin{equation}
  \Zxx_n(\ga,\gb)
=(\ga+\gb)\prod_{i=1}^{n-1}(i\ga\gb)
=(n-1)!\, \bigpar{\ga^{n}\gb^{n-1}+\ga^{n-1}\gb^{n}}.
\end{equation}
Hence there are $2(n-1)!$ tableaux in $\cSxx_n$; 
$(n-1)!$ with $n$ $\ga$'s and $n-1$ $\gb$'s, and
$(n-1)!$ with $n-1$ $\ga$'s and $n$ $\gb$'s.
See further \refS{Smax}.
(It follows that the corresponding number of \st{x} with $2n-1$ symbols
$\ga,\gb,\gam,\gd$ is 
$2^{2n}(n-1)!$, see  \cite{cd-h}.) 

By \refT{T1} and \eqref{x} below, assuming $n\ge2$, 
\begin{equation}
\P\bigpar{A\noooo=k}
=\frac{\tvx_{0,0}(n,k)}{(n-1)!}= \frac{\euler {n-1}{k-1}}{(n-1)!},
\end{equation}
and thus by \eqref{euler11} $A\noooo\eqd A_{n-2,1,1}+1$.
\end{example}
  
\begin{example}\label{E0} 
 $\gb=0$. This gives weight $0$ to any staircase tableaux with a symbol
  $\gb$, so only tableaux with just the 
symbol 
$\ga$ may occur. 
By \ref{stdiag} and \ref{stag} in the definition, the only such tableau is
the one with 
$\ga$ in every diagonal box, and no other symbols.
This limiting case is thus trivial, with $S_{n,\ga,0}$ deterministic (and
independent of the parameter $\ga$), and 
$\na=A\ngab=n$, $\nb=B\ngab=0$, 
and $Z_n(\ga,0)=\ga^n$.

This case (and the symmetric $\ga=0$) is excluded from most of our results,
but since it is trivial, the reader can easily supplement corresponding,
trivial, results for it. Note that this case occurs as a natural limiting
case when $\gb\to0$.
\end{example}

\begin{example}\label{E00} 
  $\ga=\gb=0$. This case is really excluded, since it would give weight 0 to
  every \gabst. However, we can define it as the limit as
  $\ga=\gb\to0$. This gives a non-zero probability only to \gabst{x} with a
  minimum number of symbols, \ie, with $n$ symbols on the diagonal and
  no others. There are $2^n$ such \gabst{x}, and all get the same probability,
  so $S_{n,0,0}$ is obtained by putting a random symbol in each diagonal
  box, uniformly and independently. 
This leads to a classical case and we will not discuss it any
  further.

More generally, taking the limit as $\ga,\gb\to0$ with
$\ga/(\ga+\gb)\to\rho\in\oi$ yields an \gabst{} with symbols only on the
diagonal and
each diagonal box having
symbol $\ga$ with probability $\rho$, independently of the other boxes.
(Cf.\ \refT{Toooorho}.)
\end{example}

\section{Eulerian numbers and polynomials}\label{SEuler}
As a background, we recall some standard facts about Eulerian numbers and
polynomials. 

For $a=1$, $b=0$, the recursion \eqref{vx-rec} is the standard recursion for
\emph{Eulerian numbers} $\euler nk$, 
see \eg{} \cite[Section 6.2]{CM},  \cite[\S26.14]{NIST},
\cite[A008292]{OEIS}; thus
\begin{equation}\label{v10}
  v_{1,0}(n,k) = \euler nk.
\end{equation}
(These are often defined as the number of permutations of $n$ elements with
$k$ descents (or ascents). 
See \eg{} \cite[Section 1.3]{StanleyI}, where also other
relations to permutations are given.) 
The corresponding polynomials
\begin{equation}
  P_{n,1,0}(x)=\sumkon \euler nk x^k
\end{equation}
are known as \emph{Eulerian polynomials}.

Furthermore, the cases $(a,b)=(0,1)$ and $(1,1)$ also lead to Eulerian
numbers, with different indexing:
By \eqref{vx-rec} and induction,
or by \eqref{asymm} below,
\begin{equation}\label{esymm}
    v_{0,1}(n,k) = v_{1,0}(n,n-k)=\euler n{n-k}=\euler{n}{k-1},
\qquad n\ge1,
\end{equation}
(which is non-zero for $1\le k\le n$).
Similarly, by \eqref{vx-rec} and induction,
\begin{equation}\label{v11}
    v_{1,1}(n,k)= v_{1,0}(n+1,k) = \euler {n+1}k,
\qquad n\ge0.
\end{equation}
Equivalently,
\begin{align}\label{epsymm}
P_{n,0,1}(x) &= x P_{n,1,0}(x), &   
P_{n,1,1}(x) &= P_{n+1,1,0}(x).  
\end{align}

Similarly, by the definition \eqref{tvx00} and \eqref{v11}, 
\begin{equation}\label{x}
\tvx_{0,0}(n,k)=  \euler {n-1}{k-1},
\qquad n\ge2,
\end{equation}
and by \eqref{tpn00} and \eqref{epsymm},
\begin{align}
\tP_{n,0,0}(x) =  P_{n-1,0,1}(x)
=x P_{n-1,1,0}(x).  
\end{align}

The Eulerian polynomials can also be defined by the formula
\begin{equation}\label{ep}
  \sumk (k+1)^n x^k = \frac{P_{n,1,0}(x)}{(1-x)^{n+1}}
\end{equation}
or by the (equivalent) generating function
\cite[26.14.4]{NIST} 
\begin{equation}
  \label{gep}
\sumn P_{n,1,0}(x)\frac{z^n}{n!}
=
\frac{1-x}{e^{z(x-1)}-x},
\end{equation}
both found by Euler \cite{E212}.  
(The sums converge for sufficiently small $x$ and $z$ ($|x|<1$ for
\eqref{ep});
alternatively, \eqref{ep}--\eqref{gep} can be seen as formulas for formal
power series.)

The Eulerian polynomials were introduced by Euler 
\cite{E55,E212,E352} and were used by him to calculate the sum of series.
(In particular, Euler used them to calculate the sum of the alternating series 
$\sumki(-1)^{k-1} k^n$ for $n\ge0$
\cite[p.~85]{E352}. 
This series is obviously divergent, which did not
stop Euler; in modern terminology he computed the Abel sum by taking $x=-1$
in \eqref{ep}.)
See also \cite{Hirzebruch} and \cite{Foata}.

\begin{remark}
  Notation has varied. It is now standard to define the Eulerian polynomials
  as our $P_{n,1,0}(x)$, but it was earlier common to use this multiplied by
  $x$,  \ie, our $P_{n,0,1}(x)=xP_{n,1,0}(x)$, with coefficients
  $\euler{n}{k-1}$. 
(Euler himself used both versions:
$P_{n,0,1}$ in \cite{E55} 
and $P_{n,1,0}$ in \cite{E212,E352}.)  
Similarly, notation for Eulerian numbers has varied, see \eg{}
\cite[A008292, A173018 and A123125]{OEIS}.
\end{remark}

\section{The polynomials $\pnab$}\label{Spol}

The numbers $\vab(n,k)$ and polynomials $\pnab(x)$ are defined by
\eqref{vx-rec}--\eqref{pnab-rec} for all real (or complex) $a$ and $b$, but
we are for our purposes only interested in $a,b\ge0$.
We regard $a$ and $b$ as fixed parameters, but we note that the numbers
$\vab(n,k)$ are polynomials in $a$ and $b$ (of degree exactly $n$ in the
non-trivial case $0\le k\le n$).

The case $a=b=0$ is trivial: by 
\eqref{vx-rec} or \eqref{pnab-rec} and induction
\begin{equation}
  v_{0,0}(n,k)=0
\qquad\text{and}\qquad P_{n,0,0}(x)=0
\qquad\text{for all } n\ge1.
\end{equation}

In the case when $a=0$ or $b=0$ we have the following simple relations,
generalizing the results for Eulerian numbers and polynomials in 
\eqref{esymm}--\eqref{epsymm}.

\begin{lemma}\label{LP1}
  For all $n\ge1$,
  \begin{align}
	v_{a,0}(n,k)&=av_{a,1}(n-1,k), \\
	v_{0,b}(n,k)&=bv_{1,b}(n-1,k-1),
\intertext{and, equivalently,}
	P_{n,a,0}(x)&=aP_{n-1,a,1}(x), \\
	P_{n,0,b}(x)&=bxP_{n-1,1,b}(x).
  \end{align}
\end{lemma}
\begin{proof}
  Induction, using \eqref{vx-rec} or \eqref{pnab-rec}.
\end{proof}

We collect some further properties in the following theorems.

\begin{theorem}\label{TP}
For all $a,b$ and $n\ge0$, 
  \begin{align}
	\pnab(0)&=\vab(n,0)=a^n, \label{tp0}
\\
	\vab(n,n)&=b^n, \label{tpn}
\\
  \pnab(1)&=\sumkon \vab(n,k) 
= (a+b)\rise{n}=\frac{\gG(n+a+b)}{\gG(a+b)}. \label{tp1}
  \end{align}  
Furthermore, we have the symmetry
\begin{equation}\label{asymm}
  v_{a,b}(n,k)=v_{b,a}(n,n-k)
\end{equation}
and thus
\begin{equation}\label{psymm}
  P_{n,a,b}(x) = x^n P_{n,b,a}(1/x).
\end{equation}

\end{theorem}
\begin{proof}
  Induction, using \eqref{vx-rec} or \eqref{pnab-rec}.
\end{proof}

\begin{remark}\label{RPsymm}
The symmetries \eqref{asymm}--\eqref{psymm} between $a$ and $b$ are more
evident if we define the 
homogeneous  two-variable polynomials
  \begin{equation}\label{hP}
\hP\nab(x,y)\=\sumkon \vx\ab(n,k)x^k y^{n-k} 
\end{equation}
which satisfy the recursion
\begin{equation}
  \hP_{n,a,b}(x,y)
= \Bigpar{bx+ay+xy\frac{\partial}{\partial x}+xy\frac{\partial}{\partial y}}
\hP_{n-1,a,b}(x,y),
\qquad n\ge1
\end{equation}
and the symmetry
$\hP_{n,a,b}(x,y)=\hP_{n,b,a}(y,x)$.
(Note that 
$\hP\nab(x,y)=y^nP\nab(x/y)$
and $P\nab(x)=\hP\nab(x,1)$.)

Then \eqref{t1} can be written in the symmetric form
\begin{equation}
  \begin{split}
  \E x^{A\ngab}y^{B\ngab}
&=\sumkon \P(A\ngab=k)x^k y^{n-k}
=\frac{\gG(a+b)}{\gG(n+a+b)}\hP\nab(x,y).	
  \end{split}
\end{equation}
However, we find it more convenient to work with polynomials in one variable.
\end{remark}

\begin{theorem}
  \label{TP'}
For all $a,b$ and $n\ge0$, 
  \begin{equation}\label{tp'}
  \pnab'(1)=\sumkon k\vab(n,k) = 
\frac{n(n+2b-1)}{2}(a+b)\rise{n-1}
  \end{equation}
and
  \begin{equation}\label{tp''}
	\begin{split}
  \pnab''(1)&=\sumkon k(k-1)\vab(n,k) 
\\&= 
\frac{n(n-1)(3n^2+(12b-11)n+12b^2-24b+10)}{12}(a+b)\rise{n-2}.
	\end{split}
  \end{equation}
\end{theorem}

\begin{proof}
  This can be shown by induction, differentiating \eqref{pnab-rec} once or
  twice and then taking $x=1$. We omit the details, and give instead another
  proof in \refS{Spf}.
\end{proof}

\begin{theorem}\label{TP0}
  \begin{romenumerate}[-10pt]
  \item \label{tp++}
If $a,b>0$, then $\vab(n,k)>0$ for $0\le k\le n$, and $\pnab(x)$ is a
polynomial of degree $n$ with $n$ simple negative roots.
  \item \label{tp+0}
If $a>b=0$, then $\vab(n,k)>0$ for $0\le k< n$, and $\pnab(x)$ is a
polynomial of degree $n-1$ with $n-1$ simple negative roots.
  \item \label{tp0+}
If $a=0<b$, then $\vab(n,k)>0$ for $1\le k\le n$, and $\pnab(x)$ is a
polynomial of degree $n$ with $n$ simple roots in $(-\infty,0]$;
one of the roots is $0$, provided $n>0$.
\item \label{tp00}
If $a=b=0$, then $\tvx_{0,0}(n,k)>0$ for $1\le k\le n-1$, and $\tP\noo(x)$ is a
polynomial of degree $n-1$ with $n-1$ simple roots in $(-\infty,0]$;
one of the roots is $0$, provided $n\ge2$.
  \end{romenumerate}
\end{theorem}

\begin{proof}
  \pfitemref{tp++}
Induction shows that $\vab(n,k)>0$ for $0\le k\le n$, so $\pnab$ has degree
exactly $n$. 
The fact that  all roots are negative and simple
follows from \eqref{pnab-rec},
as noted already by \citet{Frobenius} for the Eulerian polynomials; this can
be seen by the following standard argument.
Suppose, by induction, that $\pniab$ has $n-1$ simple roots $-\infty <
x_{n-1}<\dots<x_1<0$. Then $\pniab$ changes sign at each root, with a
non-zero derivative, and since
$\pniab(0)>0$ by \eqref{tp0}, we have
$\sign(\pniab'(x_i))=(-1)^{i-1}$,  $i=1,\dots,n-1$.
Since \eqref{pnab-rec} yields 
$\pnab(x_i)=x_i(1-x_i)\pniab'(x_i)$ and $x_i<0$, this implies
$\sign(\pnab(x_i))=(-1)^{i}$,  $i=1,\dots,n-1$.
Moreover, $\sign(\pnab(0))=+1$
and 
$\lim_{x\to-\infty}\sign(\pnab(x))=(-1)^n\sign(\vab(n,n))=(-1)^n$ 
by \eqref{tp0} and \eqref{tpn}.
Hence $\pnab$ changes sign at least $n$ times in $(-\infty,0)$, and thus has
at least $n$ roots there. Since $\pnab$ has degree $n$, these are all the
roots, and they are all simple.

\pfitemx{\ref{tp+0},\ \ref{tp0+}}
Follows from \ref{tp++} and \refL{LP1}.
(Alternatively, the proof above works with minor modifications.)

\pfitemx{\ref{tp00}}
Follows from \ref{tp++} and the definitions \eqref{tvx00}--\eqref{tpn00}.
\end{proof}

The proof shows also that the roots of $\pniab$ and $\pnab$ are interlaced
(except that 0 is a common root when $a=0$).
For more general results of this kind, see \eg{} \cite{WangYeh} and
\cite[Proposition 3.5]{LiuWang}.

\begin{example}
The case  $a=b=1/2$ appeared in \cite{d-hh}, see \refE{E22}.
In this case, it is more convenient to study the numbers 
$B(n,k)\=2^n\vx_{1/2,1/2}(n,k)$ which are integers
and satisfy the recursion
\begin{equation}
B(n,k)=(2k+1)B(n-1,k)+(2n-2k+1)B(n-1,k-1),
\qquad n\ge1;
\end{equation}
these are called \emph{Eulerian numbers of type B}
\cite[A060187]{OEIS}. 
The numbers $B_{n,k}$ 
seem to have been introduced by \citet[p.~331]{MacMahon} 
in number theory.
They also have combinatorial interpretations, for example as
the numbers of descents in signed permutations, \ie, in the
hyperoctahedral group
\cite{Brenti,ChowG,SchmidtS}.

Note that this case is a special case of both of the following examples.
\end{example}

\begin{example}
Franssens  \cite{pyr} studied numbers and polynomials equivalent to
the case $a=b$ of ours; more precisely, 
his $B_{n,k}(c)=2^nv_{c/2,c/2}(n,k)$, as is seen by comparing his recursion
formula to \eqref{vx-rec}, and thus his 
$B_n(x,y;c)= 2^n \hP_{n,c/2,c/2}(x,y)$,
using the notation \eqref{hP}.
The generating function in
\cite[Proposition 3.1]{pyr} thus yields (for small $|t|$)
\begin{equation}
\sumn \hP_{n,a,a}(x,y) \frac{t^n}{n!} =B(x,y,t)^{2a},
\end{equation}
with
\begin{equation}
  B(x,y,t)\=
  \begin{cases}
\frac{x-y}{xe^{-(x-y)t/2}-ye^{(x-y)t/2}}, & x\neq y	;
\\
\frac1{1-xt}, & x= y.	
  \end{cases}
\end{equation}
It would be interesting to find a similar generating function for 
$\hP_{n,a,b}(x)$ for arbitrary $a$ and $b$.
\end{example}

\begin{example}
  The case $a+b=1$ yields polynomials $P_{n,a,1-a}(x)$ generalizing the
  Eulerian polynomials (the case $a=1$, or $a=0$); they 
satisfy the following extensions of \eqref{ep}--\eqref{gep}:
\begin{equation}
  \sumk (k+a)^n x^k = \frac{P_{n,a,1-a}(x)}{(1-x)^{n+1}}
\end{equation}
and 
\begin{equation} 
\sumn P_{n,a,1-a}(x)\frac{z^n}{n!}
=
\frac{(1-x)e^{az(1-x)}}{1-xe^{z(1-x)}}.
\end{equation}
These polynomials are sometimes called (generalized)
\emph{Euler--Frobenius polynomials} and appear \eg{} in spline theory, see
\eg{} \cite{MeinardusMerz,terMorsche,Reimer:extremal,Reimer:main,Siepmann}. 
The function $P_{n,1-a,a}(x)/(x-1)^n$ was studied by \citet{Carlitz}
(there denoted $H_n(a\mid x)$).
\end{example}

We defined in \eqref{tvx00}--\eqref{tpn00} 
$\tvx_{0,0}(n,k)$ and $\tP_{n,0,0}(x)$ as substitutes for the vanishing 
$\vx_{0,0}(n,k)$ and $P_{n,0,0}(x)$. 
To justify this, we first note that these numbers and polynomials satisfy
the recursions obtained by putting $a=b=0$ in 
\eqref{vx-rec} and \eqref{pnab-rec}.
\begin{lemma}\label{L00rec}
We have
\begin{equation}\label{tvx-rec}
\tvxoo(n,k)=k\tvxoo(n-1,k)+(n-k)\tvxoo(n-1,k-1),
\qquad n\ge3,
\end{equation}
with $\tvxoo(2,1)=1$ and $\tvxoo(2,k)=0$ for $k\neq1$.
Similarly,
\begin{equation}\label{tpnab-rec}
  \tpnoo(x)
= (n-1)x  \tP_{n-1,0,0}(x)
+x(1-x)   \tP_{n-1,0,0}'(x),
\qquad n\ge3.
\end{equation}
  with $\tP_{2,0,0}(x)=x$.
\end{lemma}

\begin{proof}
  Follows easily by substituting the definitions 
\eqref{tvx00} and \eqref{tpn00} in 
\eqref{vx-rec} and \eqref{pnab-rec}.
\end{proof}

Moreover, these numbers and polynomials appear as limits as $a,b\to0$ if we
renormalize: 
\begin{lemma}\label{L00lim}
  For any $n\ge 2$ and $k\in\bbZ$ or $x\in\bbR$,
as $a,b\downto 0$,	
  \begin{align}\label{v00lim}
\frac{\vx\ab(n,k)}{a+b}&\to\tvxoo(n,k),
\\
\frac{P\nab(x)}{a+b}&\to\tpnoo(x). \label{p00lim}
  \end{align}
\end{lemma}

\begin{proof}
  We first verify \eqref{v00lim} for $n=2$ by inspection, see \refTab{Tab:vab}.
For $n>2$ we divide \eqref{vx-rec} by $a+b$, let $a,b\downto0$ and use
induction together with \eqref{tvx-rec}.

Finally, \eqref{p00lim} follows from \eqref{v00lim} by \eqref{tpn00} and
\eqref{pnab}. 
\end{proof}

\begin{remark}
More general numbers, defined by a more general version of the recursion
formula \eqref{vx-rec}, are studied in \cite{WangYeh}.  
\end{remark}

\section{Proofs of Theorems \ref{T1}--\ref{TN2}}\label{Spf}

To prove \refT{T1} we use induction on the size $n$, 
where we extend a staircase tableau of size $n-1$ by adding a
column of length $n$ to the left
and consider all possible ways of filling it out with the symbols.
This method was used, in a probabilistic context,
in \cite{d-hh} and its origins seem to go back to
\cite[Remark~3.14]{cssw}, see also \cite{cd-h}. For permutation
tableaux an analogous technique was used in \cite{ch} and \cite{hj}. 

In order to do the necessary recursive analysis, we
introduce a suitable generating function with an additional ``catalytic''
parameter that we now define. 

We say that a row of a staircase tableau is {\it indexed by $\a$} if
its leftmost entry is $\a$. 
Thus, for example,  
in the tableau depicted in \refF{fig:tab}, the first, third and
eighth rows
are indexed by $\alpha$. 
The number of rows indexed by $\a$ in a staircase tableau $S$ will be
denoted by $r=r(S)$. 

We introduce the generating function for the pair of parameters $(A,r)$:
\begin{equation}\label{dn}
  D_n(x,z)\=\sum_{S\in\css_n} \wt(S)x^{A(S)}z^{r(S)}
=\sum_{S\in\css_n} \ga^{\na}\gb^{\nb}x^{A}z^{r} .
\end{equation}
We regard $\ga$ and $\gb$ as fixed in this section, and 
for simplicity we omit them from the notation $D_n(x,z)$.
We assume that $0<\ga,\gb<\infty$.

\begin{remark}\label{Rr} 
In an \gabst, a row containing a $\gb$ must by \ref{stbd} have the $\gb$ as its
leftmost entry; hence it is not indexed by $\ga$. Conversely, a row without
$\gb$ is necessarily indexed by $\ga$.
Since no row contains more than one $\gb$,
it follows that $r=n-\nb$ \cite{d-hh}. 
We thus have $D_n(x,z)=z^n \tD_n(x,\ga,\gb/z)$ where
\begin{equation}
\tD_n(x,\ga,\gb)\=\sum_{S\in\css_n} \ga^{\na}\gb^{\nb}x^{A}=D_n(x,1).  
\end{equation}
Hence it is possible to avoid $r$ and instead argue with the simpler
$\tD_n(x,\ga,\gb)$ and a varying $\gb$. However, we find it more convenient
to keep $\ga$ and $\gb$ fixed and to use $r$ in the argument below.
\end{remark}

Trivially, $D_0(x,z)=1$ (see \refR{R0}). 

\begin{lemma}\label{L2}
$D_n$ satisfies the recursion, for $n\ge1$,
\begin{equation}\lbl{rec_gf}
D_{n}(x,z)=\a z(x-1)D_{n-1}(x,z)+(\a z+\b)D_{n-1}(x,z+\b).
  \end{equation}
\end{lemma}

\begin{proof}
Fix an \gabst{} $S$ of size $n-1$ with parameters $\na$, $\nb$, $A$, $r$, 
and consider all ways to extend it to  a tableau of size $n$ by adding a
column of length $n$ on the left and filling some boxes in it.
There are three cases,
\cf{} \cite{cd-h, d-hh}.
\begin{romenumerate}[-20pt]
\item 
We put $\a$ in the bottom box of the added column.
By \ref{stag}, no other boxes in the new column can be filled, so this gives 
a single staircase tableau of size $n$; this tableau has parameters $\na+1$,
$\nb$, $A+1$ and $r+1$, so its contribution to $D_n(x,z)$ is
\begin{equation}\label{ika1}
  \ga^{\na+1}\gb^{\nb}x^{A+1}z^{r+1} 
= \ga xz\,\ga^{\na}\gb^{\nb}x^{A}z^{r}.
\end{equation}
\item
We put $\gb$ in the bottom box of the added column; we may also put 
$\ga$ or $\gb$ in some other boxes in the new column, and we consider first
the case when we put no $\ga$, so only $\gb$'s are added.
By \ref{stbd}, we may put a $\gb$ only in the rows indexed by $\ga$ (apart
from the bottom box). For  $0\le k\le r$,
there are thus $\binom rk$ possibilities to add $k$ further $\gb$; each
choice yields a staircase tableau with parameters $\na$, $\nb+1+k$, $A$, $r-k$,
and their total contribution to $D_n(x,z)$ is 
\begin{equation}\label{ika2}
  \sum_{k=0}^r \binom rk \ga^{\na}\gb^{\nb+1+k}x^A z^{r-k}
= \ga^{\na}\gb^{\nb+1}x^A (z+\gb)^{r}.
\end{equation}

\item
We put $\gb$ in the bottom box of the added column and
$\ga$ or $\gb$ in some other boxes in the new column, including an
$\ga$. 
By \ref{stag}, we may add only one $\ga$, and it has to be the top one of
the added symbols. Again, the new symbols may (apart from the bottom box)
only be added in rows indexed by $\ga$. 
For  $1\le k\le r$,
there are thus $\binom rk$ possibilities to add $k-1$ further $\gb$ and one
$\ga$; each 
choice yields a staircase tableau with parameters $\na+1$, $\nb+k$, $A$,
$r-k+1$, 
and their total contribution to $D_n(x,z)$ is 
\begin{equation}\label{ika3}
  \sum_{k=1}^r \binom rk \ga^{\na+1}\gb^{\nb+k}x^A z^{r-k+1}
= \ga^{\na+1}\gb^{\nb}x^A z \bigpar{(z+\gb)^{r} - z^r}.
\end{equation}
 \end{romenumerate}

Combining \eqref{ika1}--\eqref{ika3}, we obtain the total contribution
from extensions of $S$ to be
\begin{equation}\label{ikam}
\ga xz\, \ga^{\na}\gb^{\nb}x^A z^r
+
 (\gb+\ga z)\ga^{\na}\gb^{\nb}x^A (z+\gb)^{r}
-\ga z\, \ga^{\na}\gb^{\nb}x^A z^r,
\end{equation}
and summing over all $S\in \css_{n-1}$ yields \eqref{rec_gf}.
\end{proof}

Iterating \eqref{rec_gf} we obtain the following, recalling that
$x\rise\ell$ denotes the rising factorial and that $a=\ga\qw$ and
$b=\gb\qw$.

\begin{lemma}\label{L3}
Assume $0<\ga,\gb<\infty$.  
  For $0\le m\le n$,
  \begin{equation}\label{l3}
D_n(x,z) = (\ga\gb)^m \sum_{\ell=0}^m
c_{m,\ell}(z) (a+bz)\rise\ell (x-1)^{m-\ell} D_{n-m}(x,z+\ell\gb),
  \end{equation}
where $c_{0,0}(z)=1$ and, for $m\ge0$, with $c_{m,-1}(z)=c_{m,m+1}(z)=0$,
\begin{equation}
c_{m+1,\ell}(z)=(\ell+bz)c_{m,\ell}(z)+c_{m,\ell-1}(z),
\quad 0\le\ell\le m+1.  
\end{equation}
\end{lemma}

\begin{proof}
  The case $m=0$ is trivial.
Suppose that \eqref{l3} holds for some $m\ge0$ and all $n\ge m$.
If $n>m$, we use \refL{L2} on the \rhs{} of \eqref{l3} and obtain
\begin{equation*}
  \begin{split}
&(\ga\gb)^{-m}D_n(x,z) 
\\&
=
\sum_{\ell=0}^m 
c_{m,\ell}(z) (a+bz)\rise\ell (x-1)^{m-\ell} 
\Bigl(\a (z+\ell\gb)(x-1)
D_{n-m-1}(x,z+\ell\gb)
\\&\hskip13em +(\a z+\ga\ell\gb+\b)D_{n-m-1}(x,z+\ell\b+\b)\Bigr)
\\&
=
\ga\gb
\sum_{\ell=0}^m c_{m,\ell}(z) (a+bz)\rise\ell (x-1)^{m+1-\ell} 
(bz+\ell)D_{n-m-1}(x,z+\ell\gb)
\\&
\quad
+
\ga\gb
\sum_{\ell=0}^m c_{m,\ell}(z) (a+bz)\rise\ell (x-1)^{m-\ell} 
(bz+\ell+a)D_{n-m-1}(x,z+\ell\b+\b)
\\&
=
\ga\gb
\sum_{\ell=0}^m (\ell+bz)c_{m,\ell}(z) (a+bz)\rise\ell (x-1)^{m+1-\ell} 
D_{n-m-1}(x,z+\ell\gb)
\\&
\quad
+
\ga\gb
\sum_{j=1}^{m+1} c_{m,j-1}(z) (a+bz)\rise j (x-1)^{m+1-j} 
D_{n-m-1}(x,z+j\b).
  \end{split}
\end{equation*}
The result for $m+1$ follows, and the lemma follows by induction.
\end{proof}

We now take $z=1$, thus forgetting $r$. 
(We will not use $r$ further. If desired, $r$ can be recovered  by \refR{Rr}.)
This yields the following formula for the generating function $D_n(x,1)$ for
$A$. 
We write $c_{n,\ell}=c_{n,\ell}(1)$.

\begin{lemma}\label{L4}
Assume $0<\ga,\gb<\infty$.  
  For $n\ge0$,
  \begin{equation}\label{l4}
D_n(x,1) = (\ga\gb)^n \sum_{\ell=0}^n
c_{n,\ell} (a+b)\rise\ell (x-1)^{n-\ell} 
  \end{equation}
where $c_{0,0}=1$ and, for $n\ge0$, with $c_{n,-1}=c_{n,n+1}=0$,
\begin{equation}\label{l4c}
c_{n+1,\ell}=(\ell+b)c_{n,\ell}+c_{n,\ell-1},
\quad 0\le\ell\le n+1.  
\end{equation}
\end{lemma}

\begin{proof}
  Take $z=1$ and $m=n$ in \refL{L3}, recalling that $D_0=1$ so the factor
  $D_{n-m}(x,z+\ell\gb)$ on the \rhs{} of \eqref{l3} disappears.
\end{proof}

We have found a formula for $D_n(x,1)$ as a polynomial in $x-1$.
We can 
identify it as $\pnab(x)$
(up to a constant factor).

\begin{lemma}\label{L5}
Assume $0<\ga,\gb<\infty$.  
  For $n\ge0$,
  \begin{equation}
D_n(x,1) = (\ga\gb)^n \pnab(x).
  \end{equation}
\end{lemma}

\begin{proof}
  Define $\hD_n(x)\=(\ga\gb)^{-n}D_n(x,1)$. Clearly,
  $\hD_0(x)=1=P_{0,a,b}(x)$.
We show that $\hD_n$ satisfies the recursion \eqref{pnab-rec},
which implies that $\hD_n=\pnab$ for all $n\ge0$ and thus completes the proof.
By \refL{L4},
\begin{equation*}
  \begin{split}
&((n+b)x+a)\hD_n(x)+x(1-x)\hD_n'(x)
\\&\quad
= \sum_{\ell=0}^n \bigpar{nx+bx+a-(n-\ell)x}
c_{n,\ell} (a+b)\rise\ell (x-1)^{n-\ell} 
\\&\quad
= \sum_{\ell=0}^n \bigpar{(\ell+b)(x-1)+\ell+b+a}
c_{n,\ell} (a+b)\rise\ell (x-1)^{n-\ell} 
\\&
=
\sum_{\ell=0}^n (\ell+b)
c_{n,\ell} (a+b)\rise\ell (x-1)^{n+1-\ell} 
+ \sum_{\ell=0}^n (a+b+\ell)
c_{n,\ell} (a+b)\rise\ell (x-1)^{n-\ell} 
\\&
=
\sum_{\ell=0}^n (\ell+b)
c_{n,\ell} (a+b)\rise\ell (x-1)^{n+1-\ell} 
+\sum_{j=1}^{n+1}
c_{n,j-1} (a+b)\rise j (x-1)^{n+1-j} 
\\&
=
\sum_{j=0}^{n+1}
c_{n+1,\ell} (a+b)\rise\ell (x-1)^{n+1-\ell} 
=\hD_{n+1}(x)
,
  \end{split}
\end{equation*}
where we used \eqref{l4c} and \eqref{l4} in the last line.
\end{proof}

\begin{proof}  [Proof of \refT{T1}]
Assume $\ga,\gb\in(0,\infty)$.
By \eqref{dn} and \eqref{zab0}, we have $D_n(1,1)=Z_n(\ga,\gb)$.
Moreover, it follows immediately from $A\ngab=A(S\ngab)$ and the definitions
\eqref{t1} and \eqref{pab} that
\begin{equation}
  g_A(x)=\sum_{S\in\css_n} x^{A(S)}\P(\sngab=S)
=\sum_{S\in\css_n} x^{A(S)}\frac{\wt(S)}{Z_n(\ga,\gb)}
=\frac{D_n(x,1)}{D_n(1,1)}.
\end{equation}
Hence, \refL{L5} yields
\begin{equation}
  g_A(x)=\frac{P\nab(x)}{P\nab(1)},
\end{equation}
which shows \eqref{t1}, using \eqref{tp1}. Extracting coefficients yields
\eqref{t1b}.

The case $\ga=\infty$ or $\gb=\infty$ follows by taking limits as
$\ga\to\infty$ ($\gb\to\infty$).
 
The case $\ga=\gb=\infty$ follows similarly by taking limits as
$\ga=\gb\to\infty$, using \refL{L00lim}.
\end{proof}

The proof above contains (as a simpler special case) the calculation of
$Z_n$ in \cite{cd-h}; we record this for completeness:

\begin{proof}
[Proof of \eqref{sim_gen} and \eqref{zab}]
Taking $x=1$ in \refL{L4} we obtain 
\begin{equation}
  Z_n(\ga,\gb)=D_n(1,1)=(\ga\gb)^n c_{n,n}(a+b)\rise n
= (\ga\gb)^n c_{n,n}(a+b)\rise n,
\end{equation}
since $c_{n,n}=1$ by \eqref{l4c} and induction.
(Alternatively, we may use \refL{L5} and \eqref{tp1}.)
This yields \eqref{zab}, and \eqref{sim_gen} follows by \eqref{ja}.
\end{proof}

\begin{proof}[Proof of \refT{TP'}.]
We assume $a,b>0$; the general case then follows since all quantities are
polynomials in $a$ and $b$.
By Lemmas \ref{L5} and \ref{L4},
for any $k\ge0$,
  \begin{equation}\label{magnus}
	\frac{\ddx^k}{\ddx x^k}\pnab(1)
= k!\, c_{n,n-k}(a+b)\rise{n-k}
  \end{equation}
(with $c_{n,\ell}=0$ for $\ell<0$). In particular, for $k=1$ we have by 
\eqref{l4c}
\begin{equation}
c_{n+1,n}=(n+b)c_{n,n}+c_{n,n-1}  
=n+b+c_{n,n-1},  
\end{equation}
and a simple induction yields
\begin{equation}
  c_{n,n-1}=\sum_{m=0}^{n-1} (m+b) = \frac{n(n+2b-1)}{2},
\end{equation}
which by \eqref{magnus} yields \eqref{tp'}.

Similarly,
\begin{equation}
  \begin{split}
  c_{n,n-2}
&=
\sum_{m=1}^n(m+b-2)c_{m-1,m-2}
\\&
=
\frac{n(n-1)(3n^2+(12b-11)n+12b^2-24b+10)}{24},	
  \end{split}
\end{equation}
which by \eqref{magnus} yields \eqref{tp''}.
\end{proof}

\begin{proof}[Proof of \refT{T2}]
  Assume first $(a,b)\neq(0,0)$. Then \eqref{t1} yields
  \begin{equation*}
\E A\ngab = g_A'(1)=\frac{P\nab'(1)}{P\nab(1)}
  \end{equation*}
and 
  \begin{equation*}
\Var A\ngab = g_A''(1)+ g_A'(1)-\bigpar{g_A'(1)}^2
=\frac{P\nab''(1)+P\nab'(1)}{P\nab(1)}
-
\frac{P\nab'(1)^2}{P\nab(1)^2}
  \end{equation*}
and the result follows from \refT{TP'} and \eqref{tp1}
(after some calculations).

The case $a=b=0$ follows by continuity.
\end{proof}

\begin{proof}[Proof of \refT{Tneg}]
  The first claim is immediate by Theorems \ref{T1} and \ref{TP0}.
This implies  \eqref{tneg} and the following claims by standard arguments:
If $g_A(x)$ has roots $-\xi_1,\dots,-\xi_n\le 0$, then, using $g_A(1)=1$,
\begin{equation}
g_A(x)=\frac{\prodin (x+\xi_i)}{\prodin (1+\xi_i)}
=\prodin \Bigpar{\frac{\xi_i}{1+\xi_i} + \frac{1}{1+\xi_i}x},
\end{equation}
which equals the \pgf{} of $\sumin \Be(p_i)$ for independent
$\Be(p_i)$ with $p_i=1/(1+\xi_i)$; this
verifies \eqref{tneg}. If $b=0$ so $g_A(x)$ has only $n-1$ roots, the same
holds with $p_n=0$. (We may then formally set $\xi_n=\infty$.)

The fact that the distribution of $A\nab$ is log-concave and thus unimodal
follows easily from \eqref{tneg} by induction; the same holds for the
sequence $\vx\ab(n,k)$, $k\in\bbZ$, by \eqref{t1b}.
\end{proof}

\begin{proof}[Proof of \refT{TCLT}]
  By \refT{Tneg},
  \begin{equation}\label{ull}
A\ngab\eqd\sumin I_i,	
  \end{equation}
with $I_i\sim\Be(p_i)$ independent. Note that then $\E A\ngab=\sumin p_i$
and $\Var A\ngab=\sumin p_i(1-p_i)$. Moreover, 
\begin{equation}
  \sumin \E |I_i-p_i|^3 \le   \sumin \E |I_i-p_i|^2=\Var A\ngab.
\end{equation}
The central limit theorem with Lyapounov's condition, see \eg{}
\cite[Theorem 7.2.2]{Gut}, shows that any sequence of sums of this type is
asymptotically normal, provided the variance tends to infinity, which holds
in our case by \refT{T2}.
\refT{T2} further shows 
\begin{align}\label{txe}
\E A\ngab&=n/2+O(1),\\
\Var A\ngab&=n/12+O(1), \label{txv}
\end{align}
which implies that the versions \eqref{tclt1} and  \eqref{tclt2}
are equivalent.

Finally, \cite[Theorem VII.3]{Petrov} shows that also a local  limit
theorem \eqref{tclt3a} 
holds for any sum of the type \eqref{ull}; again we use
\eqref{txe}--\eqref{txv} to simplify the result and obtain \eqref{tclt3b}.
\end{proof}

\begin{proof}[Proof of \refT{TN1}]
Assume first $\ga,\gb<\infty$.
  The joint \pgf{} of $(\na,\nb)$ is by definition
  \begin{equation}
	\frac{\sum_{S\in\css_n}\wt(S)x^{\na}y^{\nb}}{Z_n(\ga,\gb)}
=
	\frac{\sum_{S\in\css_n}\ga^{\na}\gb^{\nb}x^{\na}y^{\nb}}{Z_n(\ga,\gb)}
=	\frac{Z_n(\ga x,\gb y)}{Z_n(\ga,\gb)},
  \end{equation}
and \eqref{tn1} follows from \eqref{zab}.

Since $(I_i,J_i)$ defined by \eqref{tn1b} has the \pgf{} 
$\frac{bx+ay+ixy}{a+b+i}$, the distributional identity \eqref{tn1x} 
follows from \eqref{tn1}.
Thus $\E \na=\sumini\E I_i$, $\Var \na=\sumini\Var I_i$ and
$\Cov(\na,\nb)=\sumini\Cov(I_i,J_i)$, which yield \eqref{tn1e}--\eqref{tn1cov}.

The case when $\ga=\infty$ or $\gb=\infty$, or both, follows by taking limits.
\end{proof}

\begin{proof}[Proof of \refT{TN2}]
The estimates \eqref{tn2e}--\eqref{tn2cov} follow  from
\eqref{tn1e}--\eqref{tn1cov}.  

The central limit theorem \eqref{tn2a}--\eqref{tn2b}
follows from the representation \eqref{tn1x} in \refT{TN1} as in 
the proof of \refT{TCLT}; note that \eqref{tn2cov} implies
$\Cov(\na,\nb)/\log n\to0$,
which yields the 
independence of the limits in \eqref{tn2a}--\eqref{tn2b}.
\end{proof}

\section{Subtableaux}\label{Ssub}

We number the rows and columns of a \st{} by $1,\dots,n$ starting at
the NW corner (as in a matrix); the boxes are thus labelled by $(i,j)$ with
$i,j\ge1$ and $i+j\le n+1$. The diagonal boxes are $(i,n+1-i)$, 
$i=1,\dots,n$, going from NE to SW.
We denote the symbol in box $(i,j)$ of a \st{} $S$ by $S(i,j)$, with
$S(i,j)=0$ if the box is empty.

If we delete the first rows or columns from a \st, we obtain a new, smaller,
\st. For $S\in\cS_n$ and a box $(i,j)$ in $S$ (so $i+j\le n+1$),
let $S[i,j]$ be the subtableau with $(i,j)$ as its top left box, \ie, the
subtableau obtained by deleting the first $i-1$ rows and the first $j-1$
columns. Note that $S[i,j]\in\cS_{n-i-j+2}$.
(The conditions \ref{st0}--\ref{stag} are clearly satisfied.)

\begin{theorem}
  \label{Tsub}
Let $\ga,\gb\in(0,\infty]$ and $i+j\le n+1$.
The subtableau $\sngab[i,j]$ of $\sngab$ has the same distribution as 
$S_{n-i-j+2,\hga,\hgb}$, where
$\hga\qw=\ga\qw+i-1$ and
$\hgb\qw=\gb\qw+j-1$.
\end{theorem}

\begin{proof}
  Consider first the case $i=1$ and $j=2$, where we only delete the first
  (leftmost) column.
Let $S\in\css_{n-1}$. The probability that $\sngab[1,2]=S$ is proportional
to the sum of the weights of all extensions of $S$ to a \st{} in $\css_n$.
By the proof of \refL{L2}, with $x=z=1$, this sum equals, see \eqref{ikam},
\begin{equation}
  \begin{split}
	  (\gb+\ga)\ga^{\na}\gb^{\nb}(1+\gb)^r
&=
  (\gb+\ga)\ga^{\na}\gb^{\nb}{(1+\gb)}^{n-\nb}
\\&
=
  (\gb+\ga)(1+\gb)^n\ga^{\na}\Bigparfrac{\gb}{1+\gb}^{\nb},
  \end{split}
\end{equation}
so $\P(\sngab[1,2]=S)$ is proportional to $\ga^{\na}\hgb^{\nb}$ with
$\hgb\=\gb/(\gb+1)$, \ie, $\hgb\qw=\gb\qw+1$.
Hence, $\sngab[1,2]\eqd S_{n-1,\ga,\hgb}$, so the theorem holds in this case.

Next, the case $i=2$, $j=1$ where we delete the top row follows by symmetry,
see \refR{Rsymm}.

Finally, the general case follows by induction, deleting one row or column
at a time.
\end{proof}

\section{The positions of the symbols}\label{Sloc}
We have so far considered the numbers of the symbols $\ga$ and $\gb$ in a
random \gabst, and the numbers of them on the diagonal. Now we consider the
position of the symbols.
We begin by considering the symbols on the diagonal, where every box is
filled with $\ga$ or $\gb$.

\begin{theorem}\label{TposD}
  Let $\ga,\gb\in(0,\infty]$ and let $a\=\ga\qw$, $b\=\gb\qw$.
The probability 
that the $i$:th diagonal box contains $\ga$ is 
\begin{equation}
\P\bigpar{\sngab(i,n+1-i)=\ga}
=
\frac{n-i+b}{n+a+b-1},
\qquad 1\le i\le n.
\end{equation}
\end{theorem}

\begin{proof}
  If $n=1$, this follows directly from the definition and
  $\ga/(\ga+\gb)=b/(a+b)$. 

In general, we use \refT{Tsub} with $j=n+1-i$ which shows that
$\sngab[i,n+1-i]\eqd S_{1,\hga,\hgb}$ with 
$\ha\=\hga\qw=a+i-1$,
$\hb\=\hgb\qw=b+n-i$,
which yields
\begin{equation*}
\P\bigpar{\sngab(i,n+1-i)=\ga}
=\P\bigpar{S_{1,\hga,\hgb}(1,1)=\ga}
=\frac{\hga}{\hga+\hgb}
=\frac{n-i+b}{n+a+b-1}.
\qedhere
\end{equation*}
\end{proof}

The probability of an $\ga$
thus decreases linearly as we go from NE to SW, from approximately $1$ to
approximately $0$ for large $n$. Hence the top part of the diagonal contains
mainly 
$\ga$'s and the bottom part mainly $\gb$'s.
(This is very reasonable, since these choices give fewer restrictions by
\ref{stbd} and \ref{stag}.)

Non-diagonal boxes are often empty. The distribution of a given box is as
follows. 

\begin{theorem}\label{Tpos}
Let $\ga$, $\gb$ and $a,b$ be as  in Theorem~\ref{TposD}. 
The probability 
that the non-diagonal box $(i,j)$ contains $\ga$ or $\gb$ is,
\begin{align}
\P\bigpar{\sngab(i,j)=\ga}&=\frac{j-1+b}{(i+j+a+b-1)(i+j+a+b-2)},
\label{tposa}\\
\P\bigpar{\sngab(i,j)=\gb}&=\frac{i-1+a}{(i+j+a+b-1)(i+j+a+b-2)},
\label{tposb}
\intertext{and thus}
\P\bigpar{\sngab(i,j)\neq0}&=\frac{1}{i+j+a+b-1}.
\label{tpos}  
\end{align}
For $\ga=\gb=\infty$ and $i=j=1$, we interpret \eqref{tposa} and
\eqref{tposb} as $1/2$.
\end{theorem}

\begin{proof}
Consider first the case $i=j=1$.
By \refT{TN1}, the expected total number of symbols $\ga$ in $S=S\ngab$ is
\begin{equation}\label{mg1}
  \E \na = \sum_{i=0}^{n-1}\lrpar{1-\frac{a}{a+b+i}}
\end{equation}
If we delete the first column, the remaining part 
$S[1,2]$
is by \refT{Tsub} an
$S_{n-1,\ga_1,\gb_1}$ with $a_1\=\ga_1\qw=a$
and $b_1\=\gb_1\qw=b+1$; hence \refT{TN1} shows that the expected number
of symbols in $S[1,2]$ is
\begin{equation}\label{mg2}
\sum_{i=0}^{n-2}\lrpar{1-\frac{a_1}{a_1+b_1+i}}
= \sum_{i=0}^{n-2}\lrpar{1-\frac{a}{a+b+1+i}}
= \sum_{i=1}^{n-1}\lrpar{1-\frac{a}{a+b+i}}.
\end{equation}
Taking the difference of \eqref{mg1} and \eqref{mg2} we see that 
\begin{equation}\label{mg3}
  \E\bigpar{\text{$\#\ga$ in the first column}}
= 1-\frac{a}{a+b} = \frac{b}{a+b}.
\end{equation}

Now delete the first row of $S$. By \refT{Tsub}, the remainder $S[2,1]$ is
an $S_{n-1,\ga_2,\gb_2}$ with $a_2\=\ga_2\qw=a+1$ and $b_2\=\gb_2\qw=b$.
Hence \eqref{mg3} applied to this subtableau shows that
\begin{equation}\label{mg4}
  \E\bigpar{\text{$\#\ga$ in boxes $(2,1),\dots,(n,1)$}}
=  \frac{b_2}{a_2+b_2} 
=  \frac{b}{a+b+1},
\end{equation}
and taking the difference of \eqref{mg3} and \eqref{mg4} we obtain
\begin{equation}\label{mg5}
  \P\bigpar{S\ngab(1,1)=\ga} = \frac{b}{a+b} - \frac{b}{a+b+1}
=
\frac{b}{(a+b)(a+b+1)}.
\end{equation}
(This argument is valid also for $n=2$, since \eqref{mg3} holds also for
$n=1$, by \refT{TposD} or by noting that \eqref{mg2} holds, trivially, also 
for $n=1$.)

We have shown \eqref{mg5}, 
which is \eqref{tposa} for $i=j=1$.
The general case of \eqref{tposa} follows by \refT{Tsub}, 
\eqref{tposb} follows by symmetry (\refR{Rsymm})
and \eqref{tpos} follows by summing.
\end{proof}

\begin{example}  
For $2\le k\le n$, the expected total number of symbols in the boxes on the
line $i+j=k$ parallel to the diagonal is
\begin{equation}
  \sum_{i=1}^{k-1}\frac1{k+a+b-1}=\frac{k-1}{k+a+b-1}.
\end{equation}
Thus, for $k$ large there is on the average about 1 symbol on each such line
that is not too short.
(In the case $\ga=\gb=\infty$, the expectation equals 1 for every such
line.)
We do not know the distribution of symbols on the line $i+j=k$, and leave
that as an open problem. 
We conjecture that the distribution is asymptotically 
Poisson as
$n,k\to\infty$. 
\end{example}

\begin{example}  
The expected  number of $\ga$'s on the
line $i+j=k$, with $2\le k\le n$, is
\begin{equation}
  \sum_{i=1}^{k-1}\frac{j-1+b}{(k+a+b-1)(k+a+b-2)}
=\frac{(k-1)(k+2b-2)}{2(k+a+b-1)(k+a+b-2)},
\end{equation}
which is about $1/2$ for large $k$ (with equality when $\ga=\gb=\infty$).
Again, we do  not know the distribution,
but 
we conjecture that it is asymptotically Poisson as $n,k\to\infty$. 
\end{example}

We can also consider the joint distribution for several boxes. 
We consider only boxes on the diagonal, leaving non-diagonal boxes as an
open problem.
Our key tool is the following simple lemma. Compare to \refT{Tsub} with no
conditioning and (in this case) a shift of $\gb$.

\begin{lemma}  \label{Lii} 
If we condition $S\ngab$ on the bottom box $S\ngab(n,1)=\ga$,
the subtableau $S\ngab[1,2]$ obtained by deleting the first column has the
distribution of $S_{n-1,\ga,\gb}$.
\end{lemma}
\begin{proof}
  If $S$ is an \gabst{} such that the bottom box $S(n,1)=\ga$, then the first
  column is otherwise empty by \ref{stag}, and the remainder, \ie{} $S[1,2]$,
  is an arbitrary \gabst{} of size $n-1$. Introducing weights \eqref{w}, 
we see that if we condition $S\ngab$ on $S\ngab(n,1)=\ga$ and then delete
the first column, we obtain a copy of $S_{n-1,\ga,\gb}$ as asserted.
\end{proof}

The following theorem gives a complete description of the distribution of
the boxes on the diagonal. For convenience, we use a simplified notation,
letting $\snj$ be the symbol of the random $S\ngab$
in the diagonal box in \emph{column} $j$,
\ie, 
\begin{equation}
\snj\=S\ngab(n+1-j,j).
\end{equation}

\begin{theorem}\label{Tii}
  Let $\ga$, $\gb$ and $a,b$ be as  in Theorem~\ref{TposD}, 
and let $1\le j_1<\dots<j_\ell\le n$. Then
\begin{equation}
  \label{tii}
\P\bigpar{\snx{j_1}=\dots=\snx{j_\ell}=\ga}
=
\prod_{k=1}^\ell \frac{j_k-k+b}{n-k+a+b}.
\end{equation}
\end{theorem}

For $\ell=1$, this is \refT{TposD}.

\begin{proof}
We use induction on $n$. (Induction on $\ell$ is also possible.)

  If $j_1>1$, we may delete the first column, which decreases $n$ and each
  $j_k$ by 1 and, by \refT{Tsub}, increases $b$ by the same amount.
Thus \eqref{tii} follows by the inductive hypothesis.

If $j_1=1$, we use \refL{Lii} and obtain by \refT{TposD} and induction 
\begin{equation*}
  \begin{split}
\hskip2em&\hskip-2em \P\bigpar{\snx{j_1}=\dots=\snx{j_\ell}=\ga}
\\&
=
\P(\snx1=\ga)\P\bigpar{\snx{j_2}=\dots=\snx{j_\ell}=\ga\mid \snx{1}=\ga}
\\&
=
\P(\snx1=\ga)\P\bigpar{\sxx{n-1}{j_2-1}=\dots=\sxx{n-1}{j_\ell-1}
=\ga}
\\&
=\frac{b}{n+a+b-1} \prod_{k=1}^{\ell-1} \frac{j_{k+1}-1-k+b}{n-1-k+a+b},
  \end{split}
\end{equation*}
which shows \eqref{tii} in this case too.
\end{proof}

The case $\ell=2$ can also be expressed as a covariance formula.

\begin{corollary}  \label{Cii}
If $1\le j< k\le n$, then
\begin{equation*}
  \Cov\Bigpar{\ett{\snx{j}=\ga}, \ett{\snx{k}=\ga}}
=-\frac{(j-1+b)(n-k+a)}{(n+a+b-1)^2(n+a+b-2)}.
\end{equation*}
\end{corollary}

\begin{proof}
  By \refT{Tii}, the covariance is
  \begin{multline*}
\frac{j-1+b}{n-1+a+b}\cdot\frac{k-2+b}{n-2+a+b}-\frac{j-1+b}{n-1+a+b}\cdot
\frac{k-1+b}{n-1+a+b}	
\\=
\frac{j-1+b}{n-1+a+b}\lrpar{\frac{k-2+b}{n-2+a+b}-\frac{k-1+b}{n-1+a+b}},		
  \end{multline*}
and the result follows.
\end{proof}

\begin{remark}  
Barbour and Janson \cite{SJ236} studied the profile of a random permutation
tableau, which by the bijection discussed in \refS{Sperm} is equivalent to
studying the sequence of partial sums $\sum_{j=1}^k \ett{\snj=\ga}$, 
$k=1,\dots,n$, in
the case $\ga=\gb=1$; it is shown in \cite{SJ236} that after rescaling, this
sequence converges to a Gaussian process. It would be interesting to extend
this to general $\ga$ and $\gb$.
\end{remark}

\section{The case $\ga=\gb=\infty$}\label{Smax}

The limiting case $\ga=\gb=\infty$ was studied in \refE{Eoooo}, where we saw
that $\snoooo$ is a uniformly random element of
$\cSxx_n$,  the set of \gabst{} with the maximal number,
  $2n-1$, of symbols $\ga$ and $\gb$.
We study these \gabst{x} further.

\begin{lemma}\label{Lmax1}
  A \st{} $S\in\cSxx_n$ has always box $(1,1)$ filled with a symbol.
\end{lemma}
\begin{proof}
  This follows from \eqref{tpos} in \refT{Tpos}, 
taking $\ga=\gb=\infty$ and thus $a=b=0$,
which shows that the random
\st{} $S_{n,\infty,\infty}$ has a symbol in box $(1,1)$ with probability  1;
recall from \refE{Eoooo} that $S_{n,\infty,\infty}$ is uniformly distributed
in $\cSxx_n$. 

Alternatively, we can give a combinatorial proof as follows: Suppose that
$S\in\cSxx_n$ has box $(1,1)$ empty. We may replace any $\ga$ in the first
 column by $\gb$, and any $\gb$ in the first row by $\ga$, without
violating \ref{st0}--\ref{stag}, and we may then add $\ga$ (or $\gb$) in box
$(1,1)$, yielding a \st{} with one more symbol, which is a contradiction
since $\cSxx_n$ consists of the \gabst{x} with a maximum number of symbols.
\end{proof}

Given a \st{} $S\in\cSxx_n$, we let as above $S(1,1)$ be the symbol in
$(1,1)$, and we let $S'$ be the \st{} obtained by removing this symbol from
$S$.

\begin{lemma}
  \label{Lmax2}
If $S\in\cSxx_n$, then $S'$ has $n-1$ $\ga$'s and $n-1$ $\gb$'s.

More precisely, $S'$ has an $\ga$ in each column except the first, and a
$\gb$ in each row except the first.
\end{lemma}

\begin{proof}
By \ref{stag}, $S$ has at most one $\ga$ in each column; moreover, since
$(1,1)$ is filled, the first column cannot contain an $\ga$ in any other
box.
Hence, $S'$ contains no $\ga$ in the first column, and at most one $\ga$ in
every other column. Similarly,
$S'$ contains no $\gb$ in the first row and at most one in every other row.

Consequently,
$\na(S')+\nb(S')\le(n-1)+(n-1)=2n-2$. On the other hand, $S$ contains $2n-1$
symbols so $S'$ contains $2n-2$ symbols and we must have equality.
\end{proof}

Conversely, if $S_0\in\css_n$ has $n-1$ $\ga$'s and $n-1$ $\gb$'s
distributed as described in \refL{Lmax2}, then box $(1,1)$ is empty and we
may add any of $\ga$ or $\gb$ to $(1,1)$ and obtain a \st{} in $\cSxx_n$.
Let $\cSxxi_n\=\set{S':S\in\cSxx_n}$ be the set of \gabst{x} described in 
\refL{Lmax2}. The mapping $S\mapsto S'$ is thus a 2--1-map of $\cSxx_n$ onto
$\cSxxi_n$.

Given $\rho\in\oi$, we define a random \gabst{} $\snoooorho$ by picking a
random, uniformly distributed, $S'\in\cSxxi_n$ and adding a random symbol,
independent of $S'$, in box $(1,1)$, with probability $\rho$ of adding $\ga$. 
In particular, $\snooooq$ has the uniform distribution on $\cSxx_n$, \ie{},
$\snooooq=\snoooo$, see \refE{Eoooo}. 

\begin{lemma}\label{Lmax3}
  Let $\ga,\gb\in(0,\infty)$.
Then
the random tableau $\sngab$
  conditioned to have the maximum number $2n-1$ of symbols
has the distribution of $\snoooorho$ with
$\rho=\ga/(\ga+\gb)$.
\end{lemma}
\begin{proof}
  A \st{} $S\in\cSxx_n$ has weight $\ga \wt(S')$ if $S(1,1)=\ga$
and
$\gb \wt(S')$ if $S(1,1)=\gb$. Since all \st{x} $S'\in\cSxxi_n$ have the same
weight $\ga^{n-1}\gb^{n-1}$ by \refL{Lmax2}, the result follows.
\end{proof}

We have defined $\snoooo$ by letting $\ga=\gb\to\infty$. What happens if we
let $\ga\to\infty$ and $\gb\to\infty$, but with different rates?

\begin{theorem}\label{Toooorho}
  Let $\ga\to\infty$ and $\gb\to\infty$ such that
  $\ga/(\ga+\gb)\to\rho\in\oi$, and let $n\ge1$ be fixed.
Then $\sngab\dto\snoooorho$.
\end{theorem}
\begin{proof}
  The weight of every \gabst{} in $\css_n\setminus\cSxx_n$ is at most,
assuming as we may $\ga,\gb\ge1$,
\begin{equation}
\ga^n\gb^{n-2}+\ga^{n-2}\gb^{n} = o\bigpar{\ga^n\gb^{n-1}+\ga^{n-1}\gb^{n}}
=o\bigpar{Z_n(\ga,\gb)}.
\end{equation}
Hence $\P(\sngab\notin\cSxx_n)\to0$, so it suffices to consider $\sngab$
conditioned on being in $\cSxx_n$, and the result follows by \refL{Lmax3}.
\end{proof}

Thus, although the limiting distribution depends on the size of $\ga/\gb$,
it is only the distribution of the top left symbol $S(1,1)$ that is
affected; $\sngab'$ has a unique limit distribution for all
$\ga,\gb\to\infty$.
In particular, the distribution of the symbols on the diagonal has a unique
limiting distribution.

\section{An urn model}\label{Surn}    
Consider the following generalized \Polya{} urn model (an instance of
the so-called
Friedman's urn \cite{Friedman, Freedman}, which was
studied already by Bernstein
\cite{Bernstein1,Bernstein2}; see also Flajolet et al \cite{Flajolet}): 
An urn contains white and black balls. There are initially $a$ white and $b$
black balls. 
At times $1,2,\dots$, one ball is drawn at random from the urn and then
replaced, together with a new ball of the opposite colour.

Let $A_n$ [$B_n$] be the number of white [black] balls added in the $n$
first draws; we thus have $A_n+B_n=n$.
Furthermore, after $n$ draws there are $A_n+a$ white and $B_n+b$ black balls
in the urn, and thus
\begin{equation}\label{urn-rec}
  \P(A_{n+1}=k)
=
\frac{a+k}{n+a+b}\P(A_n=k)
+
\frac{n-(k-1)+b}{n+a+b}\P(A_n=k-1).
\end{equation}
Comparing \eqref{urn-rec} to \eqref{vx-rec} we find by induction
\begin{equation}\label{urn}
  \P(A_n=k)=\frac{\vx_{a,b}(n,k)}{(a+b)\rise n}.
\end{equation}
(Cf.\ \eqref{tp1}.)

In the description of the urn model, it is natural to assume that $a$ and
$b$ are integers. However, urn models of this type can easily
be extended to allow fractional balls and thus non-integer ``numbers'' of
balls, see \eg{} \cite{SJ154}.
(It is then perhaps better to talk about weights instead of numbers,
allowing balls of different weights.)
We thus may allow the initial numbers $a$ and $b$ to be any non-negative
real numbers with $a+b>0$; we still add one (whole) ball each time. (When
$a$ and $b$ are rational, with a common denominator $q$, there is also an
equivalent model starting with $qa$ and $qb$ balls and each time adding $q$
balls of the opposite colour.) Equation \eqref{urn} still holds, which
by \refT{T1} shows the following:

\begin{theorem}\label{Turn}
  Let $\ga,\gb\in(0,\infty]$, with $(\ga,\gb)\neq(\infty,\infty)$.
Then, for every $n\ge0$,
$(A\ngab,B\ngab)$ has the same distribution as $(A_n,B_n)$ in the urn model
above, starting with $a\=\ga\qw$ white and $b\=\gb\qw$ black balls.
\nopf
\end{theorem}

In this urn model, we assume $a+b>0$, since the definition assumes that we
do not start with an empty urn. We may cover the case $a=b=0$ too by any
extra rule saying which ball to add to an empty urn, for example choosing a
white or black ball at random. In any case, 
the second ball gets the opposite colour so 
the composition at time $n=2$ is
$(1,1)$, and the urn then evolves as an urn with this initial composition.
Consequently, for $n\ge2$, \eqref{urn} yields, using \eqref{tvx00},
\begin{equation}\label{urn00}
  \P(A_n=k)=\frac{\vx_{1,1}(n-2,k-1)}{2\rise n}
=\frac{\tvx_{0,0}(n,k)}{(n-1)!}.
\end{equation}
Hence, using \eqref{t100p}, 
\refT{Turn} holds in the case $\ga=\gb=\infty$ too, with this extra
interpretation. 

\begin{remark}
We similarly see that an urn started with the composition $(0,1)$ or  $(1,0)$
becomes $(1,1)$ after the first step. The relations 
$A_{n,\infty,1}\eqd A_{n-1,1,1}+1$
and 
$B_{n,\infty,1}\eqd B_{n-1,1,1}$ in \refE{Eoo1} are thus obvious for the 
corresponding urn models.
\end{remark}

Note that asymptotic normality \eqref{tclt1} is well-known for many 
generalized \Polya{} urn 
models, including this one
\cite{Bernstein1,Bernstein2,Freedman,SJ154}. We do not know any general
local limit theorems for such urn models.

\section{Permutation tableaux, alternative tableaux and tree-like tableaux}
\label{Sperm}

\emph{Permutation tableaux} (see \eg{} \cite{SW,CN,CW1,CW2,ch,hj}),
\emph{alternative tableaux} \cite{n} 
and 
\emph{tree-like tableaux} \cite{abn}
are Young diagrams (of arbitrary shape) with some symbols added according to
specific rules, see the references just given for definitions.
The size of one of these is measured by its \emph{length}, which is the sum
of the number of rows and the number of columns.

There are bijections between the \gabst{x} of size $n$, the alternative
tableaux of length $n$ and the permutation tableaux of length $n+1$
\cite[Appendix]{cw}, as well as between these and the tree-like tableaux of
length $n+2$ \cite{abn}.
(In particular, the numbers of tableaux of these four types are the same,
\viz{} $(n+1)!$. In fact, there are also several bijections between these
objects and permutations of size $n+1$
\cite{SW,CN,n,cd-h,abn}.)
 In these correspondences, the shape of the alternative tableau corresponds to 
the sequence of symbols on the diagonal of the staircase tableau, with $\ga$
and $\gb$ in the latter corresponding to vertical and horizontal steps on
the SE border of the alternative tableaux; the shapes of the permutation and
tree-like tableaux are the same with an additional first row, or additional
first row and first column, added. 

Some parameters are easily translated by these bijections; \refTab{tab:corr}
gives some important examples from \cite{cw,abn} (see these references for
definitions). 

A uniformly random tableaux of any of these types thus corresponds to a
random \st{} $S_{n,1,1}$, see \refE{E11}. For examples, this enables us to
recover several of the results for permutation tableaux in \cite{hj} from
the results above.

Furthermore, deleting the top row of a \st{} corresponds for the alternative
tableau to deleting the first step on its SE boundary; this means deleting
its last column if it is empty, and otherwise deleting the first row.
(And similarly for deleting the first column.) 
Hence, \refT{Tsub} translates to a result on subtableaux of random
alternative tableaux. 

\begin{table}
  \begin{tabular}{l|l|l|l}
alternative & permutation & tree-like & staircase \\
 tableaux &  tableaux &  tableaux &  tableaux
\\ \hline
\#rows & \#rows $-1$ & \#rows $-1$ & $A$
\\ \hline
\#columns & \#columns  & \#columns $-1$ & $B$
\\ \hline
\#free rows & \#unrestricted rows $-1$ & \#left points & $n-\nb$
\\ \hline
\#free columns & \#top 1's & \#top points & $n-\na$
\\\hline
\#$\leftarrow$ & \#restricted rows $-1$ & \#empty left cells  & $\nb-B$
\\ \hline
\#$\uparrow$ & \#top 0's & \#empty top cells & $\na-A$
\\ \hline
\end{tabular}
\caption{Some correspondences between different types of tableaux.}
\label{tab:corr}
\end{table}

\section{Staircase tableaux and the ASEP}\label{SASEP}

As mentioned in the introduction, \st{x} were introduced in
\cite{cw_pnas,cw} in connection with the 
\emph{asymmetric exclusion process (ASEP)}; as a background, we give some
details here.
The ASEP is a Markov process
describing a system of particles on a line with $n$ sites $1,\dots,n$; each
site may contain at most one particle. Particles jump one step to the right with
intensity $u$ and to the left with intensity $q$, provided the move is to a
site that is empty; moreover, 
new particles enter site 1 with intensity $\ga$ and site $n$ with intensity
$\gd$, provided these sites are empty, and particles at site 1 or 
and $n$ leave the system at rates $\gam$ and $\gb$, respectively.
(There is also a discrete-time version.)
See further \cite{cw}, which also contains references and
information on  applications and connections to other branches of science.

Explicit expressions for the steady state probabilities of the ASEP
were first given in \cite{Derrida1}.  
 Corteel and Williams \cite{cw} gave an expression using \st{x} and a more
elaborate version of the weight $\wt(S)$ and generating function for them.
For this version, we first fill the tableau $S$
by labelling  the empty boxes of $S$ with $u$'s and $q$'s as
follows: first, we fill all the boxes to the left of a $\beta$ with  $u$'s,
and all the boxes to the left of a $\delta$ with  $q$'s. Then, we fill the
remaining boxes above an $\alpha$ or a $\delta$ with  $u$'s, and the 
remaining boxes above a 
$\beta$ or a $\gamma$ with  $q$'s. When the tableau is filled, its weight,
$\wtx(S)$, is defined as
the product of labels of the boxes of $S$; this is thus a monomial of degree
$n(n+1)/2$ in $\alpha$, $\beta$, $\gamma$, $\delta$, $u$ and $q$.
For example, 
\refF{fig:filled} shows the tableau in \refF{fig:tab}
filled with $u$'s and $q$'s; its
weight is $\a^5\b^2\delta^3\g^3u^{13}q^{10}$.
\begin{figure}[htbp]
\setlength{\unitlength}{0.4cm}
\begin{picture}(10,8)(10,0)\thicklines

\put(11,0){\line(0,0){8}}
\put(12,0){\line(0,1){8}}
\put(13,1){\line(0,1){7}}
\put(14,2){\line(0,1){6}}
\put(15,3){\line(0,1){5}}
\put(16,4){\line(0,1){4}}
\put(17,5){\line(0,1){3}}
\put(18,6){\line(0,1){2}}
\put(19,7){\line(0,1){1}}

\put(11,8){\line(1,0){8}}
\put(11,7){\line(1,0){8}}
\put(11,6){\line(1,0){7}}
\put(11,5){\line(1,0){6}}
\put(11,4){\line(1,0){5}}
\put(11,3){\line(1,0){4}}
\put(11,2){\line(1,0){3}}
\put(11,1){\line(1,0){2}}
\put(11,0){\line(1,0){1}}

\put(12.25,7.25){$\alpha$}
\put(18.25,7.25){$\gamma$}
\put(11.25,0.25){$\alpha$}
\put(12.25,1.25){$\beta$}
\put(13.25,2.25){$\delta$}
\put(14.25,3.25){$\alpha$}
\put(15.25,4.25){$\delta$}
\put(16.25,5.25){$\gamma$}
\put(17.25,6.25){$\gamma$}
\put(12.25,3.25){$\delta$}
\put(13.25,5.25){$\alpha$}
\put(15.25,6.25){$\alpha$}
\put(12.25,6.25){$\beta$}

\put(11.25,1.25){$u$}
\put(11.25,2.25){$q$}
\put(11.25,3.25){$q$}
\put(11.25,4.25){$q$}
\put(11.25,5.25){$u$}
\put(11.25,6.25){$u$}
\put(11.25,7.25){$u$}

\put(12.25,2.25){$q$}
\put(12.25,4.25){$q$}
\put(12.25,5.25){$u$}

\put(13.25,3.25){$u$}
\put(13.25,4.25){$q$}
\put(13.25,6.25){$u$}
\put(13.25,7.25){$u$}

\put(14.25,4.25){$q$}
\put(14.25,5.25){$u$}
\put(14.25,6.25){$u$}
\put(14.25,7.25){$u$}

\put(15.25,5.25){$u$}
\put(15.25,7.25){$u$}

\put(16.25,6.25){$q$}
\put(16.25,7.25){$q$}

\put(17.25,7.25){$q$}

\end{picture}
\caption{
The staircase tableau in \refF{fig:tab}
filled with $u$'s and $q$'s; the weight is 
$\alpha^5\beta^2\delta^3\gamma^3u^{13}q^{10}$.}
\lbl{fig:filled}
\end{figure}
We then let $Z_n(\alpha, \beta,\gamma,\delta,q,u)$ be the total weight of
all filled staircase tableaux of size $n$, i.e. 
\begin{equation}\label{gen-gen}
 Z_n(\alpha, \beta,\gamma,\delta,q,u) = \sum_{S\in\cS_n } \wtx(S).  
\end{equation}
Obviously, $Z_n$ is a homogeneous polynomial of degree $n(n+1)/2$. 
Note that the simplified versions of $\wt(S)$ and $Z_n$  used in the present
paper are obtained by putting $u=q=1$, and that in this case $Z_n$ has the
simple form \eqref{sim_gen}.
Other special cases for which there is a simple form are discussed
in \cite{cd-h} and \cite{cssw}.
The general generating function \eqref{gen-gen} also has
connections to the Askey--Wilson polynomials, see \cite{cw,cssw}.
See also \cite{CW1} for connections between a special case and permutation
tableaux.

\newcommand\AAP{\emph{Adv. Appl. Probab.} }
\newcommand\JAP{\emph{J. Appl. Probab.} }
\newcommand\JAMS{\emph{J. \AMS} }
\newcommand\MAMS{\emph{Memoirs \AMS} }
\newcommand\PAMS{\emph{Proc. \AMS} }
\newcommand\TAMS{\emph{Trans. \AMS} }
\newcommand\AnnMS{\emph{Ann. Math. Statist.} }
\newcommand\AnnPr{\emph{Ann. Probab.} }
\newcommand\CPC{\emph{Combin. Probab. Comput.} }
\newcommand\JMAA{\emph{J. Math. Anal. Appl.} }
\newcommand\RSA{\emph{Random Struct. Alg.} }
\newcommand\ZW{\emph{Z. Wahrsch. Verw. Gebiete} }
\newcommand\DMTCS{\jour{Discr. Math. Theor. Comput. Sci.} }

\newcommand\AMS{Amer. Math. Soc.}
\newcommand\Springer{Springer-Verlag}
\newcommand\Wiley{Wiley}
\newcommand\vol{\relax}
\newcommand\jour{\emph}
\newcommand\book{\emph}
\newcommand\inbook{\emph}
\def\no#1#2,{\unskip#2, no. #1,} 
\newcommand\toappear{\unskip, to appear}

\newcommand\urlsvante{\url{http://www.math.uu.se/~svante/papers/}}
\newcommand\arxiv[1]{\url{arXiv:#1.}}
\newcommand\arXiv{\arxiv}

\def\nobibitem#1\par{}


\end{document}